\documentclass[reqno]{amsart}
\usepackage[utf8]{inputenc}
\usepackage[english]{babel}
\usepackage[margin=95pt]{geometry}

\usepackage{amsmath,amssymb,amsthm,mathtools}   
\usepackage{bbm}
\usepackage{quiver}                      
\usepackage{enumitem}

\usepackage{subfiles}

\usepackage[colorlinks=true,linkcolor=blue,citecolor=blue,urlcolor=blue]{hyperref}
\usepackage{xcolor}
\definecolor{Green4}{RGB}{0,139,0}
\definecolor{DodgerBlue4}{RGB}{16,78,139}

\usepackage{hyperref}
\hypersetup{
    colorlinks=true,
    linkcolor=Green4,
    filecolor=DodgerBlue4,
    urlcolor=DodgerBlue4,
    citecolor=Green4,
}

\theoremstyle{plain}
\newtheorem{theorem}{Theorem}[section]
\newtheorem{proposition}[theorem]{Proposition}
\newtheorem{mainthm}{Theorem}

\newtheorem{lemma}[theorem]{Lemma}
\newtheorem{corollary}[theorem]{Corollary}
\newtheorem{maincorollary}{Corollary}

\theoremstyle{definition}
\newtheorem{definition}[theorem]{Definition}
\newtheorem{notation}[theorem]{Notation}
\newtheorem{construction}[theorem]{Construction}

\newtheorem{hypothesis}[theorem]{Hypothesis}
\newtheorem{example}[theorem]{Example}

\theoremstyle{remark}
\newtheorem{remark}[theorem]{Remark}

\newcommand{\ra}{\rightarrow}
\newcommand{\lra}{\longrightarrow}

\newcommand{\rr}{\mathbb{R}}
\newcommand{\zz}{\mathbb{Z}}
\newcommand{\cc}{\mathbb{C}}

\newcommand{\C}{\mathcal{C}}
\newcommand{\D}{\mathcal{D}}
\newcommand{\M}{\mathcal{M}}
\newcommand{\Acal}{\mathcal{A}}
\newcommand{\Ncal}{\mathcal{N}}
\newcommand{\sset}{\mathcal{S}}
\newcommand{\spa}{s\mathcal{S}}

\newcommand{\set}{\mathbf{Set}}
\newcommand{\cat}{\mathbf{Cat}}
\renewcommand{\top}{\mathbf{Top}}
\newcommand{\vect}{\mathbf{Vect}}
\newcommand{\Alg}{\mathbf{Alg}}
\newcommand{\cob}{\mathbf{Bord}}
\newcommand{\CSS}{\mathcal{C\mkern-1mu S \mkern-2mu S}}
\newcommand{\CSSn}{\CSS_n}

\DeclareMathOperator{\Map}{Map}

\DeclareMathOperator{\Aut}{Aut}
\DeclareMathOperator{\Fun}{Fun}
\newcommand{\Funsm}{\Fun^{\otimes}}
\newcommand{\FunsmG}{\Fun^{\otimes}_{G}}
\newcommand{\Hom}{\mathrm{Hom}}

\DeclareMathOperator{\id}{id}
\newcommand{\colim}{\operatorname{colim}}
\newcommand{\triv}{\mathrm{triv}}  
\newcommand{\n}{{N}}                 
\renewcommand{\H}{\mathrm{Ho}}      
\newcommand{\point}{\ast}
\newcommand{\op}{\mathrm{op}}
\newcommand{\sop}{\mathrm{sop}}
\newcommand{\1}{\mathbb{1}}
\renewcommand{\coprod}{\bigsqcup}
\renewcommand{\1}{\mathbbm{1}}
\newcommand{\FR}{\mathrm{Fr}}
\newcommand{\Spin}{\mathrm{Spin}}
\newcommand{\Pin}{\mathrm{Pin}}
\newcommand{\Pic}{\mathrm{Pic}}
\newcommand{\pt}{\mathrm{pt}}
\newcommand{\Cl}{C\ell}

\usepackage{xargs}
\usepackage[colorinlistoftodos,prependcaption,textsize=small]{todonotes}
\newcommandx{\unsure}[2][1=]{\todo[linecolor=red,backgroundcolor=red!25,bordercolor=red,#1]{#2}}
\newcommandx{\change}[2][1=]{\todo[linecolor=blue,backgroundcolor=blue!25,bordercolor=blue,#1]{#2}}
\newcommandx{\info}[2][1=]{\todo[linecolor=OliveGreen,backgroundcolor=OliveGreen!25,bordercolor=OliveGreen,#1]{#2}}
\newcommandx{\improvement}[2][1=]{\todo[linecolor=DodgerBlue3,backgroundcolor=DodgerBlue3!25,bordercolor=DodgerBlue3,#1]{#2}}
\newcommandx{\thiswillnotshow}[2][1=]{\todo[disable,#1]{#2}}

\title{{Involutions on $(\infty,\MakeLowercase{n})$-categories of bordisms and reflection TQFT}s}

\author{Pedro Brunialti}
\address{University of Virginia, Charlottesville VA, USA}
\email{brunialti@virginia.edu}

\begin{document}

\begin{abstract}
We construct a strict involution, that is, a strict $\mathbb{Z}/2$-action, on the $(\infty,n)$-category $\mathbf{Bord}_n(H_n)$ of bordisms with $H_n$-structure, where $H_n$ is the symmetry type of a TQFT, using the complete $n$-fold Segal space model. Using this involution, we develop an appropriate definition of reflection TQFTs in this setting and show that they correspond to certain homotopy fixed point spaces. Finally, we apply the machinery to the example of super algebras, where, depending on the symmetry type, the homotopy fixed point data corresponds to either a stellar structure or a real structure on the algebra that the reflection theory assigns to the point. 
\end{abstract}

\maketitle

\tableofcontents

\pagebreak

\section{Introduction}

A topological quantum field theory (TQFT), as defined by Atiyah \cite{atiyah1988topological} and Segal \cite{Segal1988}, is a symmetric monoidal functor
\[Z : \cob_{(n,n-1)} \lra \C,\]
where $\cob_{(n,n-1)}$ is the $1$-category of bordisms and $\C$ is a symmetric monoidal category. A \emph{bordism} from a closed $(n-1)$-manifold $M_0$ to another $M_1$ is a compact $n$-manifold $W$ with $\partial W = M_0 \sqcup M_1$. Bordisms are the morphisms of $\cob_{(n,n-1)}$, with the closed $(n-1)$-manifolds being the objects. Usually  one works with $\C = \vect_\cc$, the category of complex vector spaces, so that if $M$ is a closed manifold, a TQFT assigns a number $Z(M) \in \cc$ to it. A theory coming from physics should assign to an $(n-1)$-dimensional manifold a Hilbert space, so something with hermitian structure, and the theory should interact well with such hermitian structure, and these are called \emph{unitary TQFTs}.

The unitarity condition of a TQFT  is not captured by Atiyah and Segal's axioms.  In \cite{FreedHopkinsReflection}, Freed and Hopkins deal with the unitarity condition  and develop a framework to make unitary expressible. More specifically, they show that unitarity can be expressed  by two components, a \emph{reflection structure} and a \emph{positivity condition}. With  this insight, they proceed in two steps. First, they show that $H_n \ra O(n)$ is a \emph{symmetry type} of a TQFT, then the $1$-category of bordisms of manifolds with $H_n$-structure $\cob_{(n,n-1)}(H_n)$  admits an involution, that is, a  symmetric monoidal functor $\beta : \cob_{(n,n-1)}(H_n) \ra \cob_{(n,n-1)}(H_n)$ such that $\beta^2\cong id$. Then, since the category of complex vector spaces $\vect_\cc$ also has an involution given by complex conjugation $\overline{(-)}: \vect_\cc \ra \vect_\cc$, they define a \emph{reflection TQFT} to be a symmetric monoidal functor 
\[
Z : \cob_{(n,n-1)}(H_n) \lra \vect_\cc
\]
equipped with  equivariance data
\[
Z(\beta X ) \xrightarrow{\cong} \overline{Z(X)}.
\]
If $Y$ is any $(n-1)$-dimensional closed $H_n$-manifold, the equivariance data allows for the construction of a non-degenerate pairing
\[
h_Y: Z(\beta Y) \otimes Z(Y) \lra \cc.
\]
Because of this pairing, they can define a reflection structure to be \emph{positive} if $h_Y$ is positive definite for all $Y$.

Freed and Hopkins define reflection positivity for \emph{non-extended theories}. A fully extended $n$-dimensional TQFT is a symmetric monoidal functor out of extended bordism categories, where  objects are points, $1$-morphisms are bordisms between objects, $2$-morphisms are bordisms between $1$-morphisms and so on up to $n$. This paper aims to expand the ideas of Freed and Hopkins to the extended case using concrete models for $(\infty,n)$-categories of bordisms developed by Calaque and Scheimbauer \cite{claudiaDamien} using $n$-fold complete Segal spaces. Here we only deal with reflection structures on extended TQFTs. In the extended case, positivity becomes data rather than a property, whose definition requires an appropriate notion of  \emph{$n$-dagger categories}, which are studied in \cite{daggerncats}.

The first step to generalize reflection structures for extended theories is to define an involution coming from the symmetry type of a TQFT on the extended category of bordisms as defined by Calaque and Scheimbauer, which is the content of the first main result. 

\begin{mainthm}[Theorem \ref{thm: equivZ2bord}]
The completion $\cob_n(H_n)$ has a strict $\zz/2$-action and is a $\zz/2$-symmetric monoidal $n$-fold complete Segal space.
\end{mainthm}

 The   strictness of the action  is a necessity for the our  model for symmetric  of $G$-symmetric monoidal category, where $G$ is a finite group, used in this paper. However, we do not want our equivariance data to be strict, as that would be too restrictive. To that end, we work with an appropriate model structure for $G$-symmetric monoidal $n$-fold complete Segal spaces and consider our theories as elements of derived mapping spaces. With this approach we get the following result characterizing reflection theories. 

\begin{mainthm}[Theorem \ref{z2fixedisreflection}]
 Let $\C$ be a $\zz/2$-symmetric monoidal $n$-fold complete Segal space. Then 
    \[
\Fun^\otimes_{\zz/2}(\cob_n(H_n),\C) \simeq \Fun^\otimes(\cob_n({H_n}),\C)^{h\zz/2}.
    \]
    where $\Fun^\otimes_{\zz/2}(\cob_n(H_n),\C)$ denotes the space of $\C$-valued reflection theories. 
\end{mainthm}

With the above theorem, we proceed to use the above result to produce the equivalent of the equivariance data from Freed and Hopkins in the extended case.

\begin{mainthm}[Theorem \ref{t:main}]
Let $\C$ be a $\zz/2$-symmetric monoidal $n$-fold complete Segal space with action
$\tau$, and let $Z$ be a $\C$-valued reflection TQFT with symmetry type $H_n$. Then for every object $Y$ of
$\cob_n(H_n)$ there is an equivalence
\[
\tau\bigl(Z(Y)\bigr) \ \simeq\ Z(\alpha Y)
\]
in $\C$, where $\alpha$ is the involution of Proposition \ref{prop:strictaction}.
\end{mainthm}

Finally, we use the obtained results to study a concrete example. We work in dimension one using the category of super algebras with complex conjugation as the target for our theories and using the group $H_1 = \Spin_1 = \{\pm1\}$ to get the following result. 

\begin{maincorollary}[Corollary \ref{cor:cliff}]
Let $A$ be a super-algebra. If $Z$ admits a reflection structure in the sense of Definition \ref{def:correctZ2}, and $Z(\pt_+)=A$, then there is an equivalence
\[
\overline{A} \ \simeq\ A^{\sop}
\]
in $s\Alg_1(\cc)$.
\end{maincorollary}

The structure on the algebra associated to the point appearing above is called a \emph{stellar structure}, and the result is consistent with those found by M\"uller and Stehouwer in \cite{MullerStehouwer}.

\subsection*{Acknowledgments}
I am very grateful to my PhD advisor, Julie Bergner, for her guidance, feedback and allowing me to pursue my interests. I would also like to thank  Eleftherios Chatzitheodoridis, for being helpful and patient with teaching me about the technical details of Bousfield localizations. I am also very thankful to  Dan Berwick-Evans for bringing the motivation for the ideas of this paper into my attention. Finally, I would also like to thank Luuk Stehouwer for being kind enough to explain further the motivation of the constructions presented here and other ways to develop the theory of reflection TQFTs.
\section{Background on  \texorpdfstring{$n$}{n}-fold complete Segal spaces}
Below we recall  some facts about the theory of complete Segal spaces. 

\subsection{Simplicial sets}

We first go over some elements of the theory of simplicial sets and model categories that we use in this paper. For more details, see \cite{goerss2009simplicial}.

\begin{definition}
    The \textit{simplex category} $\Delta$ is the category whose objects are the finite ordered sets $[n] = \{0 < 1 < \cdots < n\}$ for $n \geq 0$, and whose morphisms are the order-preserving maps between them.
\end{definition}

\begin{definition}
    A \textit{simplicial set} is a functor $X: \Delta^\op \to \set$. A morphism of simplicial sets is a natural transformation between such functors. We denote the category of simplicial sets by $\sset$. 
\end{definition}

\begin{definition}
    For $[n] \in \Delta$, the representable functor $\Delta[n] = \Hom_{\Delta}(-, [n])$ is a simplicial set called the \textit{standard $n$-simplex}. 
\end{definition}

\begin{example}
    For $\C$ a small category, the \textit{nerve} of $\C$ is the simplicial set $N\C$ whose set of $n$-simplices is  $(N\C)_n = \Hom_{\cat}([n], \C)$, the set of functors from $[n]$ to $\C$, where $[n]$ is viewed as a poset category The face and degeneracy maps are induced by the corresponding maps in $\Delta$. In particular, $\Delta[n]$ is the nerve of the category $[n]$.
\end{example}

\begin{definition}

    The identity $id_{[n]} : [n] \ra [n]$ defines a non-degenerate $n$-simplex $\sigma_n \in \Delta[n]$. The \textit{boundary}  of $\Delta[n]$ is the smallest sub-simplicial set  of $\Delta[n]$ containing $d_i(\sigma_n)$ for all $i$. The $k$-th $n$-\textit{horn} $\Lambda^k[n]$ is the smallest sub-simplicial set of $\Delta[n]$ containing $d_i(\sigma_n)$ for all $i \neq k$.

\end{definition}

\begin{definition}
    A morphism of simplicial sets $f: X \ra Y$ is a \textit{Kan fibration} if for every $n \geq 1$ and $0 \leq k \leq n$, any commutative square of the form
    \[
    \begin{tikzcd}
    \Lambda^k[n] \arrow[r] \arrow[d, hook] & X \arrow[d, "f"] \\
    \Delta[n] \arrow[r] & Y
    \end{tikzcd}
    \]
    admits a lift $\Delta[n] \ra X$ making both triangles commute. A simplicial set $X$ is a \textit{Kan complex} if the unique morphism $X \ra \Delta[0]$ is a Kan fibration.
\end{definition}

\subsection{Model categories}

Now we recall the definition of a model category, which is a category equipped with three distinguished classes of morphisms. For details see \cite{hirschhorn2003model}. 

\begin{definition}
    A \textit{category with weak equivalences} is a category $\C$ together with a subclass $W \subseteq \mathrm{Mor}(\C)$ containing isomorphisms of $\C$. The morphisms in $W$ are called \textit{weak equivalences}. Furthermore, we require that $W$ satisfies the 2-out-of-3 property: if $f: X \ra Y$ and $g: Y \ra Z$ are morphisms in $\C$ such that two of the three morphisms $f$, $g$, and $g \circ f$ are in $W$, then so is the third.
    A \textit{model category} is a category with weak equivalences $(\C,W)$  which has all limits and colimits together with two additional classes of morphisms, called \textit{fibrations} and \textit{cofibrations}. A cofibration which is also a weak equivalence is called a \textit{trivial cofibration}, and a fibration which is also a weak equivalence is called a \textit{trivial fibration}. These classes have to satisfy several axioms.
\end{definition}

\begin{definition}
   Let $\varnothing$ and $*$ denote the initial and terminal objects, respectively, of a model category  $\C$.  We say that $X$ is a \emph{fibrant object} if the unique morphism $X \ra *$ is a fibration. Dually, we say that $X$ is a \emph{cofibrant object} if the initial morphism $\varnothing \ra X$ is a cofibration. 
    
\end{definition}

\begin{theorem}\cite{Quillen:1967}
    The category of simplicial sets $\sset$ admits a model structure where the weak equivalences are the morphisms $f: X \ra Y$ such that $|f|: |X| \ra |Y|$ is a weak homotopy equivalence of topological spaces, the fibrations are the Kan fibrations, and the cofibrations are the monomorphisms.
\end{theorem}

\begin{definition}
    A \textit{Quillen adjunction} between two model categories $\C$ and $\D$ is an adjunction $F: \C \leftrightarrows \D : G$  with $F$ left adjoint such that $F$ preserves cofibrations and trivial cofibrations, or equivalently, $G$ preserves fibrations and trivial fibrations. A Quillen adjunction is a \textit{Quillen equivalence} if for every cofibrant object $X$ in $\C$ and every fibrant object $Y$ in $\D$, a morphism $f: F(X) \ra Y$ is a weak equivalence in $\D$ if and only if the corresponding morphism $f^{\#}: X \ra G(Y)$ is a weak equivalence in $\C$.
\end{definition}

\subsection{Simplicial spaces}

\begin{definition}
	\textit{A simplicial space} is a simplicial object in the category of simplicial sets $\sset$, that is, a functor $\Delta^\op \ra \sset$. A morphism of simplicial space is just a natural transformation of functors. We denote the category of simplicial spaces by $\spa$.
\end{definition}

\begin{remark}
	A simplicial space can also be regarded  as bisimplicial set, that is, a functor
	\[
	X : \Delta^\op \times \Delta^\op \lra \set.
	\]
\end{remark}

If  $X$ is a simplicial set, then there are two ways to regard  $X$ as a simplicial space. Since a simplicial space is a simplicial set in two directions we have a choice of including it on the first index or the second index. By convention, we  treat $X$ as a constant simplicial space, that is, a simplicial space whose simplicial set at level $n$ is just $X$.

 We can also  treat $X$ as a \textit{discrete} simplicial space by  letting $X_{n,m}:= X_n$, and declaring all maps $X_{n,m}\ra X_{n,k}$ to be the identity. This simplicial space is called discrete because  each level, $X_n$ is a discrete simplicial set, that is, just a set. 

\begin{notation}
	If $X$ be a simplicial set, then by abuse of notation we will write $X$ for the corresponding constant simplicial space and $X^t$ is the corresponding discrete simplicial space. The superscript $t$ stands for ``transpose", since we have that $X_{n,m} = (X^t)_{m,n}$.
\end{notation}

The category of simplicial spaces is enriched over the category of simplicial sets. To show this enrichment, we need to define what a mapping space is.

\begin{definition}
	For $X,Y$ simplicial spaces, define the simplicial set $\Map(X,Y)$ by 
	\[
	\Map(X,Y)_n = \Hom(X \times \Delta[n], Y) =\Hom(\Delta[n],\Map(X,Y)).
	\]
\end{definition}

Since the representables are $\Delta[m,n] := \Hom_{\Delta^{\op, \times 2}}(- , ([n],[m]))$, we have the following very useful identification.

\begin{proposition}
	Let $X$ be any simplicial space. Then
	\[
	\Map(\Delta[n,0], X) \cong X_n.
	\]
\end{proposition}

\subsection{Segal spaces}
In this section we recall the definition of Segal maps and Segal spaces. 
\begin{definition}
Consider the maps  $\alpha^i: [1]\ra [n]$ in $\Delta$ given by $\alpha^i(0)=i$ and $\alpha^i(1)=i+1$. Then the \textit{spine} of $\Delta[n]$ is the simplicial set
\[
G(n)=\bigcup_{i=0}^{n+1} \alpha^i(\Delta[1]) \subset \Delta[n].
\]

	Let $G(n)^t$ be $G(n)$ viewed as a simplicial space and let $X$ be a simplicial space. Then we have a map
	\[
	\Map(\Delta[n]^t, X) \lra   \Map(G(n)^t, X)
	\]
	induced by the inclusion $G(n) \hookrightarrow \Delta[n]$ that corresponds to a map of simplicial sets
	\[
	\varphi_n : X_n \lra X_1 \times_{X_0} \dots \times_{X_0} X_1.
	\]
	The map $\varphi_n$ is called the \emph{$n$-th Segal map} of $X$.
\end{definition}

\begin{example}
	If $X = \n(\C)^t$ for some category $\C$, then all of its Segal maps are isomorphisms.
\end{example}

\begin{theorem}\cite[2.4]{DWYER1993137}
	There is a proper, cofibrantly generated simplicial model structure, called the Reedy model structure, on $\spa$ in which the weak equivalences are level wise weak equivalences of simplicial sets.     Furthermore, this model strcuture coincides with the injective model structure.
\end{theorem}

This is enough to give intuition for the definition of a Segal space.

\begin{definition}
 	A \textit{Segal space} is a Reedy fibrant simplicial space $X$ such that all the Segal maps
	\[
 	\varphi_n : X_n \lra X_1 \times_{X_0} \dots \times_{X_0} X_1
	\]
	are weak equivalences of simplicial sets.
\end{definition}

We have a natural model structure on simplicial spaces coming  from that of simplicial sets. But now we have these special objects inside $\spa$, namely the Segal spaces, that  are all Reedy fibrant. Similarly, there is a model structure on $\spa$ such that the fibrant objects are exactly the Segal spaces.

\begin{theorem}\cite[7.1]{rezkCSS}
	There is a simplicial, left proper, combinatorial model structure on $\spa$ such that
	\begin{enumerate}
    	\item the cofibrations are monomorphisms;
    	\item the fibrant objects are the Segal spaces;
    	\item the weak equivalences are maps $f$ such that the induced map $\Map(f,W)$ is a weak equivalence of simplicial sets for every Segal space $W$.
	\end{enumerate}
\end{theorem}

\begin{remark}
	In this model structure, weak equivalences between Segal spaces are still levelwise weak equivalences of simplicial sets.
\end{remark}

\subsection{Complete Segal spaces}

 In this section we describe how Segal spaces can be used to  model categories $(\infty,1)$-categories.

\begin{definition}
	Let $X$ be a Segal space. Then we define:
	\begin{enumerate}
    	\item the set of objects of $X$ is $ob(X) := X_{0,0}$;
    	\item given $x,y \in ob(X)$, we define the mapping space between $x,y$ by the pullback
    	\[\begin{tikzcd}
	{\Map(x,y)} & {X_1} \\
	{\{(x,y)\}} & {X_0\times X_0;}
	\arrow["{(d_1,d_0)}", from=1-2, to=2-2]
	\arrow[from=2-1, to=2-2]
	\arrow[from=1-1, to=2-1]
	\arrow[from=1-1, to=1-2]
\end{tikzcd}\]
    	\item for $x\in ob(X)$, define $id_x \in \Map(x,y)_0$ to be $s_0(x)$.
	\end{enumerate}
\end{definition}

\begin{proposition}\cite{rezkCSS}
	If $X$ is a Segal space, then for any $x,y \in ob(X)$ the simplicial set $\Map(x,y)$ is a Kan complex. 
\end{proposition}

\begin{definition}
	Let $f,g \in \Map(x,y)$, that is $f,g \in \Map(x,y)_{0}$.  Then we say that $f,g$ are \textit{homotopic} if they lie in the same path component in the Kan complex $ \Map(x,y)$.
\end{definition}

Because of this fact, the following is well defined.

\begin{definition}\label{Homotopycatsegal}
	For $X$ a Segal space, we define the \textit{homotopy category} $\H(X)$ of $X$ to be the category with
	\begin{enumerate}
    	\item $\H(X) := X_{0,0}$, and
    	\item $\Hom(x,y) := \pi_0\Map(x,y)$.
	\end{enumerate}
\end{definition}

\begin{definition}
	Let $X$ be a Segal space. Then its space of homotopy equivalences $X_{heq}\subset X_1$ is the subsimplicial set of $X_1$ whose $0$-simplices are homotopy equivalences.
\end{definition}

\begin{remark}
	Because $s_0(x)$ is a homotopy equivalence for all $x$, the map $s_0: X_0 \ra X_1$ factors through $X_{heq}$.
\end{remark}

\begin{definition}\cite{rezkCSS}
	A Segal space $X$ is \textit{complete} if the map $s_0: X_0 \ra X_{heq}$ is an equivalence of simplicial sets.
\end{definition}

As is the case with Segal spaces, we have a model structure for complete Segal spaces.

\begin{theorem}\cite[7.2]{rezkCSS}
	There is a simplicial, left proper, combinatorial model structure on $\spa$ such that
	\begin{enumerate}
    	\item the cofibrations are monomorphisms;
    	\item the fibrant objects are the complete Segal spaces and
    	\item the weak equivalences are maps $f$ such that the induced map $\Map(f,W)$ is a weak equivalence of simplicial sets for every complete Segal space $W$.
	\end{enumerate}
\end{theorem}

\subsection{\texorpdfstring{$n$}{n}-fold complete Segal spaces}

 An $n$-fold simplicial space, that is a functor $\Delta^{\op,\times n}\ra \sset$, have access to the other categorical directions, which make them good at modeling higher categories. 

\begin{definition}
	An \textit{$n$-fold Segal space} is a Reedy fibrant functor  $W: \Delta^{\op,\times n}\ra \sset$ such that 
	\[W_{k_1, \dots, k_{i-1+1},*,k_{i+1} \dots k_n}\] 	
	is a Segal space for all $1\leq 1 \leq n$ and such that 
	 for every $1 \leq i \leq n$, and every $k_1, \dots, k_{i-1} \geqslant 0$, the $(n-i)$-simplicial space
\[
W _{k_1, \dots, k_{i-1}, 0, *, \dots, *}
\]
is essentially constant, that is, the $(n-i)$-fold simplicial above is such that every degeneracy map
$$W_{k_1,\dots,k_{i-1},0,0,\dots,0} \lra W_{k_1,\dots,k_{i-1},0,m_{i+1},\dots,m_n}$$
is a weak equivalence for all $m$'s. This is the globularity condition: the ``objects" in direction $i$ carry no higher categorical structure in the remaining directions, just a space.
\end{definition}

\begin{definition}
	A \textit{complete $n$-fold Segal space} is an $n$-fold Segal space such that for every $1\leq i \leq n$ 
\[
W _{k_1, \dots, k_{i-1}, *, 0, \dots, 0}
\]
is a complete Segal space. 

\end{definition}

\begin{definition}
	
The \textit{homotopy category $\H(W)$ of an $n$-fold Segal space} $W_{*, \dots, *}$  is the category whose objects are the elements of the set $X_{0, \dots, 0}$. For $x, y$ two objects, consider   the pullback
	\[\begin{tikzcd}
	{\Map(x,y)} & {W_1} \\
	{\{(x,y)\}} & {W_0\times W_0.}
	\arrow["{(d_1,d_0)}", from=1-2, to=2-2]
	\arrow[from=2-1, to=2-2]
	\arrow[from=1-1, to=2-1]
	\arrow[from=1-1, to=1-2]
\end{tikzcd}\]

We call it the $(n-1)$-fold Segal space of morphisms from $x$ to $y$. Now let morphisms in $\H(W)$ from $x$ to $y$ be the set of isomorphism classes of objects in $\H\left(\Map_W(x, y)\right)$ where this object is defined inductively using Definition \ref{Homotopycatsegal}.
\end{definition}

 We define a Dwyer-Kan equivalence inductively.

\begin{definition}
 A morphism $f: X \ra Y$ of $n$-fold Segal spaces is a \emph{Dwyer-Kan equivalence} if
	 \begin{enumerate}
	\item  the induced functor $\H(f): \H(X) \ra \H(Y)$ is essentially surjective.
	\item For each pair of objects $x, y \in X_{0, \ldots, 0}$, the induced morphism 
	\[\Map_X(x, y) \lra \Map_Y(f(x), f(y))\] is a Dwyer-Kan equivalence of $( n-1 )$-fold Segal spaces.
	 \end{enumerate}
\end{definition}

\begin{theorem}\cite[7.2]{comparisonJulierezk}
	 There exists a model structure on $\CSS_n$, the category of functors  $\Delta^{\op,\times n}\ra \sset$, in which the fibrant objects are exactly the $n$-fold complete Segal spaces and whose weak equivalences between fibrant objects are the Dwyer-Kan equivalences.
\end{theorem}

\section{\texorpdfstring{$G$}{ G}-monoidal structures}

In this section we define what we mean to have an $(\infty,n)$-ctageory of symmetric monoidal structure with a group action. We take the naive route of defining actions strictly which we will work out later using equivariant homotopy theory. 

\subsection{Commutative monoids and  the  category \texorpdfstring{$\Gamma$}{ Gamma}}

The ideas in this subsection are due to Segal~\cite{SEGAL1974293}. In his original work, Segal uses the category $\Gamma$, whose opposite is the category equivalent to finite pointed sets. In this paper we will not used  this category, just is opposite, so we make the convention that $\Gamma$ will already denote  the category of finite pointed sets.

\begin{definition}
    We define  $\Gamma$ to be the skeleton of the category of finite pointed sets, so   objects  are $n_+  := \{\point, 1, \dots, n\}$ and morphisms are pointed functions. 
    \end{definition}

\begin{definition}
   Let $\C$ be a category. Then a $\Gamma$-\emph{object} in $\C$ is a functor
   \[
   F: \Gamma \lra \C.
   \]
\end{definition}

The following example clarifies the intuition behind $\Gamma$-objects.

\begin{example}
    Let $M$ be a commutative monoid. We may define a $\Gamma$-set
    \[
    A_M : \Gamma \lra \set_*
    \]
    by setting
    \[
    A_M(S)  := \Hom_{\set_*}(S, M)
    \]
    where $M$ is naturally pointed by $0$, its identity element.

    For a morphism $\theta: S \ra T$ of finite pointed sets and a function of pointed sets $f: S \ra M$ we need to construct another map of pointed sets $ \theta_*(f) : T \ra M$. We define
    \[
    \theta_*(f)(t) = \begin{cases}
        0 & \text{if } t = \point \\
        \sum_{s \in \theta^{-1}(t)} f(s) & \text{if } t \neq \point.
    \end{cases}
    \]
Thus this maps sums over the pre-images of an element. Notice that the additive structure on $M$ was fundamental for the construction.

For any $n_+$ we have an identification $M(n_+) \cong M^n$.
Consider the morphism $\alpha: 2_+ \ra  1_+$ that maps both $1$ and $2$ to $1$. Then, under the identification above, $\alpha_* : M^2 \ra M$ is just the sum in the monoid.
\end{example}

The natural follow-up question is whether a given functor $A: \Gamma\ra  \set_*$ comes from a monoid. This question is the reason why we make the following definition.

\begin{definition}
   Let $X: \Gamma \ra \C$ be a functor. For any $1 \leq k \leq n$, the maps $\varphi_{n,k}: n_+ \ra 1_+$ defined by
    \[
    \varphi_{n,k}(i) = \begin{cases}
         1 & \text{if } i =k \\
         0 & \text{if } i \neq k
    \end{cases}
    \]
    induce a map $\varphi_n : X(n_+) \ra X(1_+)^n$ called the $n$-th \emph{Segal map}.
\end{definition}

\begin{example}
    Consider $A_M$ as before, let us compute its second Segal map.  The map $\varphi_{2,1}$ sends $1$ to $1$ and everything else to $0$ and $\varphi_{2,2}$ maps $2$ into $1$ and everything else to $0$. The image of these maps under $A_M$ are the components of the Segal map
    \[
\varphi_2 : \Hom_{\set_*}(2_+, M) \lra \Hom_{\set_*}(1_+,M).
    \]
If $f: 2_+ \ra M$ is in $A_M(2_+)$, then
\[
(\varphi_{2,1})_* (f) (1) = \sum_{s \in \varphi_{2,1}^{-1}(1)} f(s) = f(1).
\]
Similarly,
\[
(\varphi_{2,2})(f)(1)  = f(2).
\]
Thus, $\varphi_2$ is the natural isomorphism $A_M(2_+) \cong M^2$ appearing in the correspondence $A_M(n_+) = M^n$.
\end{example}

This  example gives us the intuition we need to understand why  we  can use $\Gamma$ to express commutative monoids.  Indeed, suppose that we have a functor $A : \Gamma \ra \set_*$such that $A(0_+) = *$ and such that all Segal maps $\varphi_n : A(n_+) \ra A(1_+)^n$ are isomorphisms, then $A \cong A_M$, where $M = A(1_+)$ and let $ \alpha_* \circ \varphi_2^{-1}: A(1_+)^2 \ra A(1_+)$ be the addition. One can check that $A_M \cong A$.

\begin{remark}
    In the example above, the target category is the category of pointed sets, but we could have used the category of sets. In fact, because $0_+$ is initial in $\Gamma$ and $A(0_+)$ is a point, all the maps induced by  $A$ preserve the base points $0_+ \ra A(S)$.
\end{remark}

\subsection{Symmetric monoidal \texorpdfstring{$n$}{n}-fold Segal spaces}

Using the idea that $\Gamma$-objects satisfying  Segal conditions encode commutative monoids in a category, we  can make the following definition. 

\begin{definition}
	 A \emph{ special $n$-fold complete Segal space} is a functor 
	\[
	A: \Gamma \lra \CSSn
	\]
    such that $A(\underline{k}_+)$ is an $n$-fold complete Segal space for all $k$  and all the Segal maps are weak equivalences of $n$-fold complete Segal spaces. Furthermore, if it is the case that $A(0_+)$ is weakly equivalent to a point, then we call it a  a \emph{symmetric monoidal $n$-fold complete Segal space}. 
\end{definition}

\begin{remark}
    Just as we did for monoids, we think of the actual monoidal $(\infty,n)$-category being the category $A(1_+)$ and the maps in $\Gamma$ encoding the monoidal product in it.  
\end{remark}

\subsection{\texorpdfstring{$G$}{G}-symmetric monoidal \texorpdfstring{$n$}{n}-fold complete Segal spaces}

In what follows $G$ is a finite group, and $BG$ its associated groupoid. 

\begin{definition}\label{GammaGAction}
    Let $\C$ be a category. A \emph{$\Gamma$-object with $G$-action} in $\C$ is a functor
    \[
    X : \Gamma \times BG \lra \C.
    \]
    We write $\rho^X_g := X(\mathrm{id}, g)$.
\end{definition}

\begin{remark}\label{BHGGamma}
    Bergner and Hackney \cite{JulieHackney} encode the same structure using a category $G\Gamma$. Its objects are
    \[
    n_G := \bigvee_{g\in G} n_g,
    \]
    and its morphisms are generated by the diagonal maps $\bigvee f$, applying the same $f : n_+ \ra m_+$ on every copy, together with the action maps $g\cdot k_h := k_{gh}$. For example, for $G = \zz/2$, $n_{\zz/2}$ is
    \[
    \{n_0,\dots, 1_0, 0, 1_1, \dots, n_1\},
    \]
    and the action swaps $k_0$ and $k_1$.

This category $G\Gamma$ is  the category obtained from $\Gamma \times BG$ by quotienting out the automorphisms of $0_+$.      Since every zero map factors through $0_+$, precomposition with $q$, the quotient functor, gives an isomorphism between functors $G\Gamma \ra \C$ with $X(0_G) = \point$ and functors $\Gamma\times BG \ra \C$ with $X(0_+) = \point$. We work with $\Gamma \times BG$ since we only require $X(0_+)$ to be weakly contractible and because it makes the techniques used later on simpler.
\end{remark}

\begin{definition}\label{GSegalMaps}
    Consider a functor $X : \Gamma \times BG \ra \C$. The \emph{Segal maps}  are the ones of the  underlying $\Gamma$-object,
    \[
    \varphi_n : X(n_+) \lra X(1_+)^n,
    \]
    with components $X(\varphi_{n,k}, e)$. Since
    \[
    (\varphi_{n,k},e)\circ(\mathrm{id},g) = (\mathrm{id},g)\circ(\varphi_{n,k},e)
    \]
    in $\Gamma\times BG$, the map $\varphi_n$ is $G$-equivariant for the diagonal action on $X(1_+)^n$.
\end{definition}

As was the case for monoids, if $M$ is a commutative monoid with a $G$-action by monoid homomorphisms, we get a functor
\[
A^G_M : \Gamma \times BG \lra \set
\]
given by $A^G_M(n_+) = \Hom_{\set_*}(n_+, M)$ on objects, by the formula for $A_M$ on morphisms of $\Gamma$, and with $g \in G$ acting by postcomposition with the action of $g$ on $M$.   The category of functors $X : \Gamma \times BG \ra \set$ such that $X(0_+) = \point$ and all Segal maps are isomorphisms is equivalent to the category of commutative monoids with a $G$-action by monoid homomorphisms.

All this discussion motivates the following definition.

\begin{definition}\label{CSSnMon}
    A \emph{$G$-symmetric monoidal $n$-fold complete Segal space} is a functor
    \[
    X : \Gamma \times BG \lra \CSS_n
    \]
    such that $X(k_+)$ is a complete $n$-fold Segal space for all $k$, $X(0_+)$ is weakly equivalent to a point and all Segal maps are equivalences of $n$-fold complete Segal spaces. Equivalently, it is a symmetric monoidal $n$-fold complete Segal space in the sense of the previous subsection together with a strict $G$-action by symmetric monoidal automorphisms.
\end{definition}

\section{Bordisms as an \texorpdfstring{$n$}{n}-fold complete Segal spaces}

We use this structure we just defined to  show that if  there is an appropriate extension of the symmetry group of the bordism category as specified in~\cite[Section 4]{FreedHopkinsReflection}, then there is an extension of $\cob_n(H_n)$, in the sense of~\cite{claudiaDamien}, to a $(\zz/2)\Gamma$-object in $n$-fold complete Segal spaces.

\subsection{The bordism category}

The basis for the construction of the $n$-fold complete Segal space of $n$-bordisms is a simpler Segal space which will be used to identify where to cut our cobordisms. To do  this, we define a category of intervals internal to topological spaces and then take the associated simplicial space.

\begin{definition}
Define the category $\mathfrak{Int^c}$ by letting objects be intervals $[a,b] \subset \rr$ and by declaring that there is a morphism $[a_0,b_0] \ra [a_1,b_1]$ if and only if $a_0 \leq a_1$ and $b_0 \leq b_1$
\end{definition}  

\begin{proposition}
    The category $\mathfrak{Int^c}$ is internal to $\top$, that is, both the set of morphisms and objects have the natural structure of a topological space and the source and target maps are continuous.
\end{proposition}
\begin{proof}  The objects of $\mathfrak{Int^c}$ are in bijection with the set $\{(a,b)  \ | \ a < b\} \subset \rr^2$ and morphisms are in bijection with the set $\{(a_0,a_1,b_0,b_1) \in \rr^2 \ | \ a_i < b_i \ \text{and } \ a_0 \leq a_1, \ b_0 \leq b_1 \} \subset \rr^4$, thus we can topologize them as subspaces. With these topologies source, target and identity maps are continuous. \end{proof}   

\begin{corollary}
    The category $\mathfrak{Int^c}$ is internal to $\sset$. 
\end{corollary}
\begin{proof} Note that both the space of morphisms and the space of objects are smooth manifolds and that all structure maps are smooth, thus we can get a simplicial set by considering the smooth singular complexes of these spaces. \end{proof}  

\begin{lemma}
A category internal to $\sset$ induces a simplicial space. 
\end{lemma}
\begin{proof} If   \[ \C = \begin{tikzcd}
	X & Y
	\arrow["id"{description}, from=1-1, to=1-2]
	\arrow["t", curve={height=-9pt}, from=1-2, to=1-1]
	\arrow["s"', curve={height=9pt}, from=1-2, to=1-1]
\end{tikzcd}\] is a category internal to $\sset$, then for $n\geq 0$ 
 \[\C_n =  \begin{tikzcd}
	X_n & Y_n
	\arrow["id_n"{description}, from=1-1, to=1-2]
	\arrow["t_n", curve={height=-9pt}, from=1-2, to=1-1]
	\arrow["s_n"', curve={height=9pt}, from=1-2, to=1-1]
\end{tikzcd}\] is a category and thus $\C $ corresponds to a functor  $\C: \Delta^\op \ra \cat$ given by $[n] \mapsto \C_n$. To get a simplicial space, we can consider the composition 
\[
\Delta^\op \xrightarrow{\C} \cat \xrightarrow{\n} \sset 
\]
where $\n$ is the nerve functor. \end{proof} 

\begin{definition}
    The \textit{simplicial space of intervals} $\mathrm{Int}_\bullet$ is the simplicial space  corresponding to 
\[
\Delta^\op \times \Delta^\op \xrightarrow{swap} \Delta^\op \times \Delta^\op \xrightarrow{\mathfrak{Int}^c} \set 
\]
    
\[
\mathrm{Int^n}_{\bullet,\dots,\bullet} : \left(\Delta^\op\right)^{n} \lra \sset
\]
by 
\[
\mathrm{Int^n}_{\bullet,\dots,\bullet}([k_1],\dots,[k_n]) = \mathrm{Int}_\bullet([k_1])\times \dots \times  \mathrm{Int}_\bullet([k_n]).
\]
\end{definition}

\begin{remark}
    In the definition of $\mathrm{Int}_\bullet$ we have the swap map so that the composition in the corresponding Segal space corresponds to the composition of intervals. 
\end{remark}

\begin{proposition}\cite[4.9]{claudiaDamien}
    The  functor $\mathrm{Int^n}_{\bullet,\dots,\bullet}$ is a levelwise contractible $n$-fold complete Segal space.
\end{proposition}

\begin{remark}
    By definition, a $0$-simplex of $\mathrm{Int}_k$ is a sequence of  intervals $I_0 \leq I_1 \leq \dots \leq I_k$, where we write $I = [a,b] \leq [a',b'] = I'$ if $a\leq a'$ and $b \leq b'$. An $l$-simplex consists of smooth maps
\[
|\Delta^l|_e \rightarrow \mathbb{R},  \ \ \  s \mapsto a_j(s), b_j(s)
\]
for $j=0, \dots, k$ such that for every $s \in |\Delta^l |_e$, the following inequalities hold:
\[
\begin{aligned}
a_i(s)<b_i(s), & \text { for } i=0, \dots, k \\
a_{i-1}(s) \leq a_i(s), & \text { and } \\
b_{i-1}(s) \leq b_i(s) & \text { for } i=1, \dots, k
\end{aligned}
\]
Here $|\Delta^l |_e$ represents the extended standard simplex. We denote an $l$-simplex of $\mathrm{Int}_k$ by $\left(I_0(s) \leq \cdots \leq I_k(s)\right)_{s \in |\Delta^l|_e}$.

\end{remark}

\begin{definition}
Fix $k \geqslant 0$. The map of spaces
\[
 B: \operatorname{Int}^n_k \longrightarrow \operatorname{Int}^n_0 \\
\]
defined by 
\[
\underline{\overline{I}} = (I^i_0 \leq \dots \leq I^i_{k_i})_{1\leq i \leq n} \mapsto B(\underline{\overline{I}})=B(\underline{a}^i(s), \underline{b}^i(s))_{s \in |\Delta^l|_e}= \prod_i\left(a^i_0(s), b^i_k(s)\right)_{s \in |\Delta^l|_e}
\]
is called the boxing map. We see $B(\underline{\overline{I}}(s))_{s \in |\Delta^l|_e}$ embedded in $\rr^n \times |\Delta^l|_e$ as $\bigcup_{s \in |\Delta^l|_e} B(\underline{\overline{I}}(s)) \times \{s\}$.
\end{definition}

In what follows,  $V$  denotes a finite dimensional real vector space. Such  a vector space can be used to construct the $n$-fold complete Segal spaces of bordisms. 

		
		

\begin{definition} An $l$-simplex of the simplicial set $(P\cob_n^V)_{k_1, \dots, k_n}$ consists of tuples $(M, \underline{\overline{I}}(s)=\left(I_0^i(s) \leq \dots \leq I_{k_i}^i(s)\right)_{s \in|\Delta^l|_e}$ such that

\begin{enumerate}

\item $\underline{\overline{I}}(s)=\left(I_0^i(s) \leq \cdots \leq I_{k_i}^i(s)\right)_{1 \leq i \leq n} $ is an $l$-simplex in $\operatorname{Int}_{k_1, \dots, k_n}^n$,

\item  $M$ is a closed and bounded $(n+l)$-dimensional submanifold of $V \times B(\underline{\overline{I}}(s))_{s \in|\Delta^l|_e} \subset V \times \mathbb{R}^n \times|\Delta^l|_e$ such that 

(a) the composition $\pi: M \hookrightarrow V \times B(\underline{\bar{I}}(s))_{s \in|\Delta^l|_e} \rightarrow B(\underline{\bar{I}}(s))_{s \in|\Delta^l|_e}$ of the inclusion with the projection is proper,

(b) its composition with the projection onto $|\Delta^l|_e$ is a submersion $M \rightarrow |\Delta^l|_e$ which is trivial outside $|\Delta^l| \subset|\Delta^l|_e$, and

\item for every $S \subseteq\{1, \dots, n\}$, let $p_S: M \xrightarrow{\pi} B(\underline{\bar{I}}
(s))_{s \in|\Delta^l|_e} \subset \mathbb{R}^n \times|\Delta^l|_e \xrightarrow{\pi_S} \mathbb{R}^S \times|\Delta^l|_e$ be the composition of $\pi$ with the projection $\pi_S$ onto the $S$-coordinates. Then for every $1 \leq i \leq n$ and $0 \leq j_i \leq k_i$, at every $x \in p_{\{i\}}^{-1}\left(\bigcup_{s \in|\Delta^l|_e} I_{j_i}^i(s) \times\{s\}\right)$, the map $p_{\{i, \dots, n\}}$ is submersive.
	
\end{enumerate}
\end{definition}

\begin{proposition}\cite[5.19]{claudiaDamien}
	The simplicial sets  $(P\cob_n^V)_{k_1, \dots, k_n}$ assemble into an $n$-fold Segal space $P\cob_n^V$.
\end{proposition}

\begin{definition}\cite[5.25]{claudiaDamien}
Fix  some countably infinite-dimensional real vector space, say $\rr^{\infty}$. Let $P\cob_n$ be the colimit $\colim_{V \subset \rr^\infty} P\cob_n^V$. The  $n$-fold complete Segal space of $n$-bordisms $\cob_n$ is the completion of $P\cob_n$ to a $n$-fold complete Segal space.
\end{definition}

\subsection{The \texorpdfstring{$\Gamma$}{Gamma}-object structure}

Now that we have category of bordisms we need to show that it is in fact a symmetric monoidal $n$-fold complete Segal space in the sense of Definition~\ref{CSSnMon}.

\begin{definition}\label{def:PBordm}
Let $V$ be a finite-dimensional vector space. For every $k_1, \dots, k_n$, let $(P\cob_n^V\langle m\rangle)_{k_1, \dots, k_n}$ be the collection of tuples
\[
(M_1, \dots, M_m,(I_0^i \leq \dots \leq I_{k_i}^i)_{i=1, \dots, n}),
\]
where each $(M_\beta,(I_0^i \leq \dots \leq I_{k_i}^i)_{i=1, \dots n})$ is an element of $(P\cob_n^V)_{k_1, \dots, k_n}$ and $M_1, \dots, M_m$ are disjoint. 

It can be shown that  simplicial sets $P\cob_n^V\langle m\rangle)_{k_1, \dots, k_n}$ assemble into an $n$-fold Segal space. Moreover, similarly to the definition of $\cob_n$, we can take the colimit over all $V \subset \mathbb{R}^{\infty}$ and complete to get an $n$-fold complete Segal space $\cob_n\langle m\rangle$.
\end{definition}

With this definition we can state the following  theorem. 

\begin{theorem}\cite[7.2]{claudiaDamien}\label{thm:Gammaobject}
Let
\[
\Gamma  \lra \CSS_n \] 

\[
\langle m\rangle  \mapsto {P\cob_n\langle m\rangle},
\]

where to a morphism $f:\langle m\rangle \rightarrow\langle k\rangle$ in $\Gamma$ we assign the map

\[P\cob_n\langle m\rangle  \longrightarrow P\cob_n\langle k\rangle\]

\[
\left(M_1, \ldots, M_m, I^{\prime} s\right)  \longmapsto\left(\coprod_{\beta \in f^{-1}(1)} M_\beta, \ldots, \coprod_{\beta \in f^{-1}(k)} M_\beta, I^{\prime} s\right).
\]

Then, the postcompsistion of this functor with the completion functor $R: \CSS_n \ra \CSS_n$ assignment is functorial and endows $\cob_n$ with a symmetric monoidal structure.
\end{theorem}

\section{Bordisms with extra structure} 

When dealing with bordisms categories one often wants to consider extra structure on the manifolds and bordisms, for instance, the Cobordism Hypothesis is about \emph{framed manifolds and bordisms}. In this section we define such structures.

 \subsection{Tangent structures}

The structures we are interested in are \emph{tangent structures}, that is, geometrical data that we can endow the tangent bundle of a manifold with. The prototypical example of such structure is \emph{orientation}, which simply consists a smooth and compatible choice of orientation for each tangent space of the manifold. As we know, not all manifolds admit an orientation and the same is true for the structure. For instance, a \emph{framing} of a manifold is a choice of a trivialization of the tangent bundle, which is very restrictive. A common way to encode such information in a more concise and more amenable to homotopy theory is to use the language of classifying spaces and reduction of structure groups of a bundle, and that is the route we follow.

\begin{definition}
Given a principal $G$-bundle $ P \to M$ and a left action of $G$ on manifold $F$,  the \emph{associated bundle}  with fiber $F$ is
\[P \times_G F := (P \times F)/G\]
where $G$ acts on $P \times F$ by
\[g \cdot (p, f) = (pg^{-1}, gf).\]

If $\rho: H \ra G$ is a morphism of Lie groups, then $H$ acts on $G$. Thus, if $Q\ra M$ a principal $H$-bundle, we can form associated bundle which in this case is actually  principal $G$-bundle $Q_\rho$ given by 
\[
\rho(Q) := (Q \times G)/H
\]
where $H$ acts freely on $Q\times G$ by
\[(q,g)\cdot h = (q\cdot h, \rho(h)^{-1}\cdot h).\]
The action of $G$ on $Q_\rho$ is given by   
\[
[q,g]\cdot k := [q,gk].\]
\end{definition}

\begin{definition}
	Let $P$ be a principal $G$-bundle and $\rho : H \ra G$ a morphism of Lie groups. Then a \emph{reduction of the structure group} of $P$ to $H$ is pair $(Q,\phi)$ where $Q$ is a principal $H$ bundle and $\phi$ is an isomorphism  
	\[
\phi : \rho(Q) \xrightarrow{ \ \sim \ } P
	\]
	of principal $G$-bundles. 
\end{definition}

Every manifold $M$ has an associated principal $O(n)$-bundle, the \emph{orthonormal frame bundle} $\FR(M)$. We can always assume that $M$ has a metric, and with that we can construct a bundle $\FR(M) \ra M$ where the fiber of $p \in M$ is the space of all orthonormal bases of $T_pM$, and the action of $O(n)$ comes from composing a bases with an orthogonal map. If we let $O(n)$ act of $\rr^n$ in the natural, we get an associated $O(n)$-bundle $\rr^n\times_{O(n)} \FR(M)$ and this bundle is isomorphic to $TM$. This motivates the following definition.

\begin{definition}
	Let $\rho: G \ra O(n)$ be a morphism of Lie groups and let $M$ be a manifold. Then a\emph{ $G$-structure} on $M$ is a reduction of  the structure group  of $\FR(M) \ra M$ to $G$, where $\FR(M)$ is the $O(n)$-bundle of orthonormal frames of $M$. 
\end{definition}

\begin{example}
	Consider the inclusion $SO(n) \hookrightarrow O(n)$. Then an $SO(n)$-structure on an $n$-dimensional manifold $M$ is a reduction of the structure group of $\FR(M)$ to $SO(n)$, which corresponds to an orientation on $M$. 
\end{example}

To make this data more amenable to homotopy theory, we can look at \emph{classifying spaces}. For any Lie group $G$ we have a principal $G$-bundle $EG \ra BG$, called the \emph{universal $G$ bundle}, such that $EG$ is contractible and has a free $G$-action. Then it is true that principal $G$-bundle $P\ra M$ classified by a map $f: M \ra BG$ in the sense there there exists a map $f: M\ra BG$ which is unique up to homotopy and such that $P$ is isomorphic to the the pullback bundle 

\[\begin{tikzcd}
	{EG \times_f M} & EG \\
	M & BG.
	\arrow[from=1-1, to=1-2]
	\arrow[from=1-1, to=2-1]
	\arrow["\lrcorner"{anchor=center, pos=0.125}, draw=none, from=1-1, to=2-2]
	\arrow[from=1-2, to=2-2]
	\arrow["f"', from=2-1, to=2-2]
\end{tikzcd}\]

\begin{lemma}\label{lem:Gstructure}
	A $G$-structure on $M$ is equivalently a pair $(\triv,f)$ where $f : M \ra BG$ and $triv$ is an isomorphisms of vector bundles 
	\[
	\triv : TM \lra f^* E_\rho
	\]
	where $E_\rho$ is the vector bundle 
	\[\begin{tikzcd}
	{E_\rho} & {EO(n)} \\
	BG & {BO(n)}
	\arrow[from=1-1, to=1-2]
	\arrow[from=1-1, to=2-1]
	\arrow["\lrcorner"{anchor=center, pos=0.125}, draw=none, from=1-1, to=2-2]
	\arrow[from=1-2, to=2-2]
	\arrow["\rho"', from=2-1, to=2-2]
\end{tikzcd}\]
where $EO(n)$ is the universal $O(n)$-vector bundle. 
\end{lemma}

\begin{proof} Let $(Q,\phi)$ be a $G$-structure on $M$. Because $Q$ is a principal $G$-bundle, it has a classifying map $f: M \ra BG$. By definition, $\phi$ is an isomorphism
\[
\phi: \rho(Q) \lra \FR(M)
\]
which induces an isomorphism
\[
\rr^n \times_{O(n)} \rho(Q) \lra \rr^n \times_{O(n)} \FR(M) \cong TM.
\]
By definition $\rr^n \times_{O(n)} \rho(Q) \cong f^* E_\rho$. \end{proof} 

This lemma helps with the following definition explains the following definition of the category of Bordisms with $G$-structures.

\begin{definition}\label{def:PBordstructured}
	Let $V$ be a finite dimensional real vector space. Define  the $n$-simplicial $P\cob_n^V(G)$ to be the $n$-simplicial space whose $l$-simplices of  $P\cob_n^V(G)_{k_1\dots,k_n}$ consist of tuples $(M, f, triv, \underline{\bar{I}}(s)=(I_0^i(s) \leq \dots \leq I_{k_i}^i(s))_{s \in|\Delta^l|_e}$ such that
\begin{enumerate}
	\item  $\overline{\underline{I}}=(I_0^i(s) \leq \dots \leq I_{k_i}^i(s))_{ s\in |\Delta^l|_e}$ is an $l$-simplex in $\operatorname{Int}_{k_1, \dots, k_n}^n$,
	\item $M$ is a closed and bounded $(n+l)$-dimensional submanifold of $V \times B(\overline{\underline{I}})$ such that 

	(a) the composition $\pi: M \hookrightarrow V \times B(\underline{\bar{I}}) \ra B(\underline{\bar{I}})$ of the inclusion with the projection is 
	proper,

	(b) its composition with the projection onto $|\Delta^l|_e$, $\pi_l: M \ra |\Delta^l|_e$, is a submersion,

	(c) $(f, triv ): \operatorname{ker}(D \pi_l: T M \ra  T|\Delta^l|_e) \ra f^* \gamma_G$ is a fiberwise linear isomorphism.

	\item for every $S \subseteq\{1, \dots, n\}$, let $p_S: M \xrightarrow{\pi} B(\underline{\bar{I}}) \subseteq \mathbb{R}^n \times|\Delta^l|_e \xrightarrow{\pi_S} \mathbb{R}^S \times|\Delta^l|_e$ be the composition of $\pi$ with the projection $\pi_S$ onto the $S$-coordinates. Then for every $1 \leq i \leq n$ and $0 \leq j_i \leq k_i$, at every $x \in p_{\{i\}}^{-1}(I_{j_i}^i(s) \times\{s\})$, the map $p_{\{i, \dots, n\}}$ is submersive.
\end{enumerate}
\end{definition}

As before, it can be shown(\cite{claudiaDamien}), that $\colim_{V \subset \rr^\infty} P\cob_n^V(G)$ is an $n$-fold Segal space  which can be completed to an $n$-fold complete Segal space $\cob_n(G)$.

\subsection{The \texorpdfstring{$\zz/2$}{C2}-action}

In this section we show that under certain hypotheses on the structure group of the bordism category, the $n$-fold complete Segal space $\cob_n(G)$ inherits the structure of a symmetric $\zz/2$-monoidal $n$-fold Segal space (Theorem \ref{thm: equivZ2bord}).

The idea on how to get this action is the following. Suppose that we have a split shot exact sequence of Lie groups
\[\begin{tikzcd}
	1 & G & {\hat G} & Q & 1.
	\arrow[from=1-1, to=1-2]
	\arrow[from=1-2, to=1-3]
	\arrow[from=1-3, to=1-4]
	\arrow["s"', shift right, curve={height=6pt}, from=1-4, to=1-3]
	\arrow[from=1-4, to=1-5]
\end{tikzcd}\]

The splitting $s: Q \ra \hat G$ induces an action of $Q$ on $G$. Now suppose that $P$ is a principal $G$-bundle, so that we have a right action  $(p,g) \mapsto p\cdot g$ on $P$. For an element $q \in Q$ we let $\sigma_q \in \Aut(G)$ be the corresponding automorphism of $G$, and we can define the twisted action 
\[
\rho_q : G \times P \lra P 
\]
\[
(p,g) \mapsto g \cdot \sigma_q(g).
\]
This still defines a principal $G$-bundle, $P_q$. If the classifying map of $P$ is  $f: M \ra BG$, then the classifying map of $P_q$ is 
\[
f\circ B(\sigma_q) : M \xrightarrow{f}   BG \xrightarrow{B(\sigma_q)} BG 
\]
and thus $Q$ acts on the $G$-bundles over $M$.

\begin{definition}\cite[2.4]{FreedHopkinsReflection}
	 The symmetry type of a topological quantum field theory  is a pair $\left(H_n, \rho_n\right)$ of a compact Lie group $H_n$ and a homomorphism $\rho_n: H_n \rightarrow O(n)$ whose image contains $S O(n) \subset O(n)$. The kernel $K$ of $\rho_n$ is called the group of internal symmetries.

	 Since a symmetry type of a TQFT is in particular a homomorphism $\rho_n : H_n \ra O(n)$, we have a well defined bordism $n$-fold complete Segal space $\cob_n(H_n)$. 
\end{definition}

\begin{notation}\label{not:models}
We will fix the following for the rest of the paper. For a Lie group $G$ let $EG$ be the bar construction and $BG = EG/G$, and thus $E$ and $B$, so that $B(\varphi\psi) = B(\varphi)B(\psi)$ and $B(\id) = \id$, and likewise for $E$. For a morphism $\rho \colon G \ra O(n)$ set
\[
\rho(EG)  :=  EG\times_G O(n), \ \ \ \ \gamma_G  :=  \rr^n \times_{O(n)} \rho(EG),
\]
where $G$ acts on $O(n)$ through $\rho$. With this we get the pullback of Lemma
\ref{lem:Gstructure}. We write the points of $\gamma_G$ as $[v,[e,g]]$ with $v \in \rr^n$,
$e \in EG$ and $g \in O(n)$, subject to the two relations
\begin{equation}\label{eq:relations}
[v,[e,g]] \ = \ \bigl[v,[ek, \rho(k)^{-1}g]\bigr] \ \ \ \ (k \in G),
\qquad
[v,[e,g]] \ = \ \bigl[a^{-1}v,\ [e,\ ga]\bigr] \ \ \ \ (a \in O(n)).
\end{equation}
Different models of $BG$ and $\gamma_G$ are related by canonical equivalences and give
equivalent bordism categories. We use this one as it is convenient to our constructions. 
\end{notation}
 
\begin{notation}\label{not:hath}
Let $(H_n,\rho_n)$ be a symmetry type. Fix a hyperplane reflection $\varsigma \in O(n)$,
so $\varsigma^2 = \id$, and let $\hat h \in \hat H_n$ be the splitting element given by
\cite[Prop.~3.13]{FreedHopkinsReflection} associated to it, so that
\begin{equation}\label{eq:hath}
\hat\rho_n(\hat h)  = (-1,\varsigma), \ \ \ \  \hat h^{2} =  1 .
\end{equation}
We denote $c_{\hat h}(k) := \hat h k \hat h^{-1}$ for the induced automorphism of $H_n$ and $c_\varsigma(a) := \varsigma a \varsigma^{-1}$ for conjugation on $O(n)$. Then
\begin{equation}\label{eq:intertwine}
\rho_n \circ c_{\hat h}  =  c_\varsigma \circ \rho_n, \ \ \ \
c_{\hat h}^{2} =  \id_{H_n},
\end{equation}
where the second equality holds because $\hat h^{2} = 1$.
\end{notation}
 
\begin{remark}\label{rmk:choiceofh}
The splitting of \cite[Prop.~3.13]{FreedHopkinsReflection} depends on the choice of
hyperplane reflection. If $\hat h$ and $\hat h'$ are the splitting elements attached to
$\varsigma$ and $\varsigma'$ then $\hat h^{-1}\hat h'$ lies in
$\hat\rho_n^{-1}(\{1\}\times O(n)) = H_n$, so $c_{\hat h}$ and $c_{\hat h'}$ differ by an
inner automorphism and the resulting actions are isomorphic. We fix
$\varsigma$ and $\hat h$ for the remainder of the paper.
\end{remark}

\begin{construction}\label{con:lambda}
Because the bar $E$ is a functor, the automorphism $c_{\hat h}$ of $H_n$ induces a map
\[
Ec_{\hat h} : EH_n \ra EH_n, \ \ \ \
Ec_{\hat h}(e\cdot k)  = Ec_{\hat h}(e)\cdot c_{\hat h}(k) \ \ \ \ (k \in H_n),
\]
which is $Bc_{\hat h}: BH_n \ra BH_n$  after quotiented. Define
\[
\lambda : \gamma_{H_n} \ra \gamma_{H_n}, \ \ \ \
\lambda\bigl([v,[e,g]]\bigr)  :=  \bigl[\varsigma^{-1}v,
[Ec_{\hat h}(e) \varsigma^{-1}g\varsigma]\bigr].
\]
\end{construction}
 
\begin{lemma}\label{lem:lambdawelldef}
The map $\lambda$ is a well-defined fiberwise linear isomorphism covering the map
$Bc_{\hat h} : BH_n \ra BH_n$. Thus it induces an isomorphism of vector bundles,
\[
\lambda^{\sharp} : \gamma_{H_n} \xrightarrow{\ \sim\ }
(Bc_{\hat h})^{*}\gamma_{H_n}.
\]
\end{lemma}
 
\begin{proof}
We check that $\lambda$ respects the two relations \eqref{eq:relations}.
 
For the first, replace $(e,g)$ by $(ek,\rho_n(k)^{-1}g)$ with $k \in H_n$. By
\eqref{eq:intertwine} and $\varsigma^2 = \id$ we have
$\varsigma^{-1}\rho_n(k)\varsigma = c_\varsigma(\rho_n(k)) = \rho_n(c_{\hat h}(k))$, and
hence
\[
\bigl(Ec_{\hat h}(ek),\ \varsigma^{-1}\rho_n(k)^{-1}g\varsigma\bigr)
\ = \ \bigl(Ec_{\hat h}(e)\,c_{\hat h}(k),\
\rho_n\bigl(c_{\hat h}(k)\bigr)^{-1}\varsigma^{-1}g\varsigma\bigr),
\]
which is the first relation applied to $c_{\hat h}(k)$. So the class is unchanged.
 
For the second, replace $(v,g)$ by $(a^{-1}v, ga)$ with $a \in O(n)$. Then
\[
\bigl(\varsigma^{-1}a^{-1}v, \varsigma^{-1}ga\varsigma\bigr)
 =  \Bigl(\bigl(\varsigma^{-1}a\varsigma\bigr)^{-1}\varsigma^{-1}v,
\bigl(\varsigma^{-1}g\varsigma\bigr)\bigl(\varsigma^{-1}a\varsigma\bigr)\Bigr),
\]
which is the second relation applied to $\varsigma^{-1}a\varsigma \in O(n)$.
 
On the base, $\lambda$ induces the map $[e]\mapsto[Ec_{\hat h}(e)]$ on $BH_n = EH_n/H_n$,
which is $Bc_{\hat h}$ by Notation \ref{not:models}. On a fiber it acts by
$v \mapsto \varsigma^{-1}v$, a linear isomorphism. The final statement is the universal
property of the pullback.
\end{proof}
 
\begin{lemma}\label{lem:lambdainvolutive}
The map $\lambda$ satisfies $\lambda^2 = \id_{\gamma_{H_n}}$.
\end{lemma}
 
\begin{proof}
Since $E$ is a functor, $Ec_{\hat h}\circ Ec_{\hat h} = E(c_{\hat h}^{\,2}) = E(\id) =
\id$ by \eqref{eq:intertwine}, and $\varsigma^2 = \id$. Hence
\[
\lambda^2\bigl([v,[e,g]]\bigr)
\ = \ \bigl[\varsigma^{-2}v,\ [\,E(c_{\hat h}^{\,2})(e),\
\varsigma^{-2}g\varsigma^{2}\,]\bigr]
\ = \ \bigl[v,[e,g]\bigr].
\qedhere
\]
\end{proof}

\begin{construction}\label{con:alpha}
For an $l$-simplex $\sigma = (M, f, \triv, \bar I)$ of
$(P\cob^V_n(H_n))_{k_1,\dots,k_n}$ as in Definition \ref{def:PBordstructured}, set
\[
\alpha(\sigma) \ := \ \bigl(M,\ Bc_{\hat h}\circ f,\ (f^{*}\lambda^{\sharp})\circ\triv,\
\bar I\bigr).
\]
\end{construction}
 
\begin{lemma}\label{lem:alphawelldef}
For a simplex $\sigma$, $\alpha(\sigma)$ is again an $l$-simplex of $(P\cob^V_n(H_n))_{k_1,\dots,k_n}$.
\end{lemma}
 
\begin{proof}
Conditions (1), (2a), (2b) and (3) of Definition \ref{def:PBordstructured}  only cares about $\bar I$, the submanifold $M$, and the projections $\pi$, $\pi_l$ and $p_S$, none of
which $\alpha$ changes, so we only need to check (2c).Lemma \ref{lem:lambdawelldef} gives an isomorphism
$\lambda^{\sharp} :  \gamma_{H_n}\ra (Bc_{\hat h})^{*}\gamma_{H_n}$ of bundles over
$BH_n$, so pulling back along $f$ yields an isomorphism
\[
f^{*}\lambda^{\sharp} \colon f^{*}\gamma_{H_n} \xrightarrow{\ \sim\ }
f^{*}(Bc_{\hat h})^{*}\gamma_{H_n} \ = \ (Bc_{\hat h}\circ f)^{*}\gamma_{H_n},
\]
the equality is given by the functoriality of pullback. Composing the fiberwise isomorphism $\triv$ with an isomorphism gives a fiberwise isomorphism, as required.
\end{proof}
 
\begin{lemma}\label{lem:alphasimplicial}
The map $\alpha$ commutes with the face and degeneracy maps in all $n$ simplicial directions,
with the spatial face and degeneracy maps, and with the colimit over
$V \subset \rr^\infty$.
\end{lemma}
 
\begin{proof}
 Simplicial maps act on $\sigma = (M,f,\triv,\bar I)$ by reindexing $\bar I$ and, for a morphism $g \colon [m]\ra[l]$ of $\Delta$, by pulling $M \ra |\Delta^l|_e$ back along $|g|$ and restricting $f$ and $\triv$. Since $\alpha$ alters $f$ only by postcomposition with the fixed map $Bc_{\hat h}$, and $\triv$ only by composition with $f^{*}\lambda^{\sharp}$, which is natural in $f$, both commute with restriction:
\[
(Bc_{\hat h}\circ f)\big|_{g^{\Delta}M} = Bc_{\hat h}\circ\bigl(f|_{g^{\Delta}M}\bigr),
 \ \ \ \
\bigl(f^{*}\lambda^{\sharp}\bigr)\big|_{g^{\Delta}M}
= \bigl(f|_{g^{\Delta}M}\bigr)^{*}\lambda^{\sharp}.
\]
The maps in the $n$ simplicial directions only reindex $\bar I$ and restrict $M$,
which $\alpha$ does not change. Finally $\alpha$ does not act on $V$, so it commutes
with the maps $P\cob^V_n(H_n)\ra P\cob^{V'}_n(H_n)$ for $V\subseteq V'$ and thus with
the colimit.
\end{proof}
 
\begin{lemma}\label{lem:alphamonoidal}
 The map $\alpha$ preserves the collections of Definition \ref{def:PBordm} and commutes with the
structure maps of the $\Gamma$-object of Theorem \ref{thm:Gammaobject}. That is,
\[
\alpha\Bigl(\bigsqcup_\beta M_\beta\Bigr) \ = \ \bigsqcup_\beta \alpha(M_\beta).
\]
\end{lemma}
 
\begin{proof}
A structure on a disjoint union is a map out of the disjoint union, and
$Bc_{\hat h}\circ(f\sqcup f') = (Bc_{\hat h}\circ f)\sqcup(Bc_{\hat h}\circ f')$.
Also, $(f\sqcup f')^{*}\lambda^{\sharp} = f^{*}\lambda^{\sharp}\sqcup
(f')^{*}\lambda^{\sharp}$. Disjointness of the $M_\beta$ is unaffected, since $\alpha$
does not move the submanifolds. The structure maps of Theorem \ref{thm:Gammaobject} take
disjoint unions indexed by fibers of a map in $\Gamma$, so they commute with $\alpha$.
\end{proof}
 
\begin{proposition}\label{prop:strictaction}
We have $\alpha^2 = \id$ so that  $\alpha$ defines a strict $\zz/2$-action on $P\cob_n(H_n)$. The same formula defines a strict $\zz/2$-action on
$P\cob_n(H_n)\langle m\rangle$ for every $m$, compatibly with the $\Gamma$-structure, so
that $P\cob_n(H_n)$ is a $(\zz/2)\Gamma$-object.
\end{proposition}
 
\begin{proof}
For the classifying map,
\[
Bc_{\hat h}\circ Bc_{\hat h} \ = \ B\bigl(c_{\hat h}^{\,2}\bigr)
\ = \ B(\id_{H_n}) \ = \ \id_{BH_n},
\]
by functoriality of $B$ and \eqref{eq:intertwine}. For the trivialization, note that $\lambda$ is a self-map of the total space $\gamma_{H_n}$, and that by construction $\alpha$ composes $\triv$ with $\lambda$ read through the identification of Lemma \ref{lem:alphawelldef}. Applying $\alpha$ twice therefore composes $\triv$ with $\lambda\circ\lambda$, which is $\id_{\gamma_{H_n}}$ by Lemma \ref{lem:lambdainvolutive}; and  since $Bc_{\hat h}^{\,2}f = f$, the resulting strcuture is  $\triv$ itself. The remaining data $M$ and $\bar I$ are unchanged, so $\alpha^2 = \id$ on each simplex.
 
By Lemmas \ref{lem:alphawelldef} and \ref{lem:alphasimplicial}, $\alpha$ is an
a morphism of the $n$-uple simplicial space $P\cob_n(H_n)$, and $\alpha^2 = \id$ makes it a functor out of $B\zz/2$. 
\end{proof}
 
\begin{theorem}\label{thm: equivZ2bord}
The completion $\cob_n(H_n)$ has a strict $\zz/2$-action and is a
$\zz/2$-symmetric monoidal $n$-fold complete Segal space.
\end{theorem}
 
\begin{proof}
Since fibrant replacement  is a functor, it carries the functor $B\zz/2 \ra \CSS_n$ of Proposition \ref{prop:strictaction} to one with value $\cob_n(H_n)$, and the same for each
$\langle m\rangle$. The Segal maps of the resulting $(\zz/2)\Gamma$-object are  equivalences because the ones in Theorem \ref{thm:Gammaobject} are preserved by completion. 
\end{proof}

\section{Reflection TQFTs}

Now that we have a $\zz/2$-action on bordisms, we can start to study TQFTs which
respect such action. First, let us define what  a TQFT is.
\begin{definition}
    Let $\cob_n(G)$ be a bordism category and let $\C$ be a symmetric monoidal
    $(\infty,n)$-category. Then a \emph{fully extended $G$-topological quantum field theory}
    $Z : \cob_n(G) \ra \C$ is a natural transformation
\[\begin{tikzcd}
	\Gamma & {} & {\CSS_n} \\
	& {}
	\arrow[""{name=0, anchor=center, inner sep=0}, "{\cob_n(G)\langle - \rangle}", curve={height=-18pt}, from=1-1, to=1-3]
	\arrow[""{name=1, anchor=center, inner sep=0}, "\C"', curve={height=18pt}, from=1-1, to=1-3]
	\arrow["Z", between={0.2}{0.8}, Rightarrow, from=0, to=1]
\end{tikzcd}\]
Just as is the case for the symmetric monoidal structure, by abuse of notation we
regard the actual functor as $Z = Z(1_+) : \cob_n(G) \ra \C$.
\end{definition}

To make a $\zz/2$-equivariant version of this, we can simply use the action obtained in
the last section. We end up with a definition of the following form.

    Let $\cob_n^{\zz/2}(H_n)$ be a $\zz/2$-bordism category and let $\C$ be a
    $\zz/2$-symmetric monoidal $(\infty,n)$-category. Then a fully extended
    $\zz/2$-equivariant $H_n$-topological quantum field theory
    $Z : \cob_n^{\zz/2}(H_n) \ra \C$ is a natural transformation
\[\begin{tikzcd}
	(\zz/2)\Gamma & {} & {\CSS_n} \\
	& {}
	\arrow[""{name=0, anchor=center, inner sep=0}, "{\cob^{\zz/2}_n(H_n)\langle - \rangle}", curve={height=-18pt}, from=1-1, to=1-3]
	\arrow[""{name=1, anchor=center, inner sep=0}, "\C"', curve={height=18pt}, from=1-1, to=1-3]
	\arrow["Z", between={0.2}{0.8}, Rightarrow, from=0, to=1]
\end{tikzcd}\]

This definition would be the first thing one writes down, and it does
define an equivariant theory. It is not, however, the object we want. The reason why this is not the object we want is that so far we have treated the action strictly, but we need to deal with equivariant data up to equivalence. To fix this problem we use derived mapping spaces in an appropriate model category of of equivariant $(\infty,n)$-categories.  A derived mapping space depends on neither variable up to weak equivalence. Below we work out the details of carrying out this construction in the derived setting and end up with correct object, which is Definition \ref{def:correctZ2}.

\subsection{Equivariant functors as homotopy fixed points}\label{sec:hfp}

We would like to understand the space of equivariant theories. In this subsection we apply tools from equivariant homotopy theory to get the following theorem   in the language of complete $n$-fold Segal spaces. 

\begin{theorem}\label{thm:main}
Let $G$ be a finite group and let $\C,\D$ be $G$-symmetric monoidal
$n$-fold Segal spaces. Then there is a weak equivalence
\[
  \FunsmG(\C,\D)\simeq\Funsm(\C,\D)^{hG},
\]
where $\Funsm(\C,\D)$ is the derived mapping space as in
Definition~\ref{def:funsm} and $G$ acts on it by the conjugation action of Construction~\ref{con:enrich},
\[
  g\cdot f=\rho^{\D}_{g}\circ f\circ\bigl(\rho^{\C}_{g}\bigr)^{-1}.
\]
\end{theorem}

For this section $G$ denotes a finite group.

\begin{hypothesis}\label{hyp:main}
Let $\M$ be a model category satisfying:
\begin{enumerate}[label=(M\arabic*),leftmargin=2.8em]
  \item\label{M1} $\M$ is simplicial, combinatorial and left proper;
  \item\label{M2} every object of $\M$ is cofibrant.
\end{enumerate}
\end{hypothesis}

\begin{lemma}\label{lem:Mn}
The category $\CSS_n$ satisfies \ref{M1} and \ref{M2}.
\end{lemma}
\begin{proof}
    The model structure on $\CSS_n$ is obtained via  a left Bousfield localization of the reedy model structure, thus all objects are still cofibrant. Taking a Bousfield localization preserves $\ref{M1}$ by  \cite[Thm.~4.1.1]{hirschhorn2003model}.
\end{proof}

\begin{definition}\label{def:Msm}
Let $\Ncal:=\Fun(\Gamma,\CSS_{n})_{\mathrm{inj}}$ and let $S$ denote the set of Segal maps
tensored with a set of generating cofibrations of $\CSS_{n}$. Define
\[
  \CSS^\otimes_n := L_{S}\Ncal .
\]
\end{definition}

\begin{lemma}\label{lem:Msm}
The category $\CSS^\otimes_n$ satisfies \ref{M1} and \ref{M2}, and its fibrant objects are precisely the special
$\Gamma$-objects taking values in  $n$-fold complete Segal spaces.
\end{lemma}

\begin{proof}
The cofibrations are levelwise monomorphisms, so \ref{M2} holds by Lemma~\ref{lem:Mn}. By  \cite[Thm.~4.1.1]{hirschhorn2003model}, the localization of s simplicial model category is simplicial. By definition the $S$-local objects are the special $\Gamma$-objects. 
\end{proof}

\begin{definition}\label{def:funsm}
Let $R$ be a functorial fibrant replacement in  $\CSS^\otimes_n$. For $C,D\in\CSS^\otimes_n$ set
\[
  \Funsm(\C,\D):=\Map_{\CSS^\otimes_n}(\C,R\D),
\]
the derived mapping space. Since we assume \ref{M2} to be satisfied, we do not need to take a cofibrant replacement of the domain. 
\end{definition}

Now we have to worry about the $G$-action.

\begin{lemma}\label{lem:curry}
There is an isomorphism of categories

 \[\Fun\bigl(\Gamma\times BG,\ \CSS_n\bigr)\cong\Ncal^{BG}.\]

\end{lemma}

\begin{proof}
    The result follows from the adjunction 
    \[
\Fun(\C\times \D, \mathcal{E}) \cong \Fun(\C, \Fun(\D,\mathcal{E})).    \]
\end{proof}

We equip $\Ncal^{BG}$, and more generally $\M^{BG}$ for $\M$ as in
Hypotheses~\ref{hyp:main}, with the projective model structure, whose weak equivalences and
fibrations are created by the forgetful functor $U:\M^{BG}\to\M$ induced by
$*\hookrightarrow BG$. It exists because $\M$ is combinatorial and $BG$ is finite.  Its
generating (trivial) cofibrations are the maps $G\otimes f$ for $f$ a generating (trivial)
cofibration of $\M$, where
\[
  G\otimes(-):\M\longrightarrow\M^{BG},\qquad G\otimes A=\coprod_{g\in G}A,
\]
denotes the free $G$-object functor, left adjoint to $U$.

\begin{definition}\label{def:Gsm}
A $G$-symmetric monoidal complete $n$-fold Segal defines an object of
$(\CSS^\otimes_n)^{BG}$. For  $\C,\D \in (\CSS^\otimes_n)^{BG}$, we write
\[
  \FunsmG(\C,\D)
\]
for the derived mapping space in $(\CSS^\otimes_n)^{BG}_{\mathrm{proj}}$. 
\end{definition}

The following compatibility allows us to work with $\CSS^\otimes_n$ directly rather than carrying the
Segal condition through every subsequent argument.

\begin{proposition}\label{prop:loc}
Let $\Acal$ be a left proper combinatorial simplicial model category and let $S$ be a set of
maps between cofibrant objects. Then
\[
  \bigl(L_{S}\Acal\bigr)^{BG}_{\mathrm{proj}}=L_{G\otimes S}\bigl(\Acal^{BG}_{\mathrm{proj}}\bigr),
  \qquad G\otimes S:=\{G\otimes s\ :\ s\in S\}.
\]
\end{proposition}

\begin{proof}
Both model structures are left Bousfield localizations of $\Acal^{BG}_{\mathrm{proj}}$: they
have the same cofibrations, since projective cofibrations depend only on the cofibrations of
the target category and $L_{S}\Acal$ has the same cofibrations as $\Acal$, and their weak
equivalences contain the projective weak equivalences. It therefore suffices to show that
they have the same fibrant objects.

An object $Y$ is fibrant in $(L_{S}\Acal)^{BG}_{\mathrm{proj}}$ if and only if $UY$ is
fibrant in $\Acal$ and $S$-local. It is fibrant in
$L_{G\otimes S}(\Acal^{BG}_{\mathrm{proj}})$ if and only if $UY$ is fibrant in $\Acal$ and
$Y$ is $(G\otimes S)$-local. The free-forgetful adjunction is simplicially enriched, so for
every $A\in\Acal$
\[
  \Map_{\Acal^{BG}}(G\otimes A,\ Y)\cong\Map_{\Acal}(A,\ UY),
\]
naturally in $A$; hence for $s\in S$ the map $\Map_{\Acal^{BG}}(G\otimes s,Y)$ is a weak
equivalence if and only if $\Map_{\Acal}(s,UY)$ is. Thus $Y$ is $(G\otimes S)$-local if and
only if $UY$ is $S$-local, and the two localizations have the same fibrant objects. Two left
Bousfield localizations of a fixed left proper combinatorial model category with the same fibrant objects coincide, since their homotopy categories can be identified with full subcategories of localized fibrant objects in the original homotopy category \cite[Proposition 10.2]{lawson2020introductionbousfieldlocalization}. 
\end{proof}

\begin{definition}
    \label{con:enrich}
Let $\M$ satisfy Hypotheses~\ref{hyp:main} let  $K$  be a $G$-simplicial set with action  
$\lambda^{K}$. Then the category of functors $\M^{BG}$ is tensored, cotensored and enriched over $\sset^G$ with $\underline{\Map}(X,Y)_{p} :=\Map_{\M}(UX,UY)_{p}$,      $g\cdot f :=\rho^{Y}_{g}\circ f\circ\bigl(\rho^{X}_{g}\otimes\id_{\Delta[p]}\bigr)^{-1}$, $ X\otimes K :=UX\otimes K$ with the diagonal action and  $Y^{K} :=(UY)^{K}$ with the conjugation action
where $\rho^{X}$ and $\rho^{Y}$ denote the given actions, $(UY)^{K}$ is the cotensor in $\M$,
and $g$ acts on it by $\rho^{Y}_{g}$ in the base together with $(\lambda^{K}_{g})^{-1}$ in
the exponent. These satisfy the two-variable adjunction

\begin{equation}\label{eq:adj}
  \Hom_{\M^{BG}}\bigl(X\otimes K,\ Y\bigr)\cong\Hom_{\sset^G}\bigl(K,\ \underline{\Map}(X,Y)\bigr).
\end{equation}
\end{definition}

\begin{lemma}\label{lem:strict}
If $X,Y\in\M^{BG}$, then there is an equality of simplicial sets

  \[\Map_{\M^{BG}}(X,Y)=\bigl(\underline{\Map}(X,Y)\bigr)^{G}\]

where the right-hand side denotes fixed points of the conjugation action.
\end{lemma}

\begin{proof}
Consider $K=\Delta[p]$ with the trivial action and apply Equation \ref{eq:adj}. Then  by the definition of the simplicial enrichment of a diagram category, $\Map_{\M^{BG}}(X,Y)_{p}=\Hom_{\M^{BG}}(X\otimes\Delta[p],Y)$, and a morphism of $G$-objects is exactly a $G$-fixed element of the underlying hom-set for the conjugation action.
\end{proof}

\begin{remark}\label{rem:tautology}
Note that above equality holds for any $X$ and $Y$, and on the right we have actual fixed points. The left hand side is  the derived mapping space only if $X$ is cofibrant and $Y$ fibrant in $\M^{BG}_{\mathrm{proj}}$. 
Projective cofibrations of $BG$-diagrams are generated by $G\otimes f$ with $f$ a generating cofibration, so the projectively cofibrant objects are retracts of cell complexes such that the cells carry a free $G$-action. Thus, \ref{M2} does not guarantee that objects of $\M^{BG}$ are
projectively cofibrant. For instance, the point with trivial $G$-action is not cofibrant. 
\end{remark}

As is usual in the literature, let $EG$ denote the simplicial set  $EG_{n}=G^{\times{n+1}}$ with
the diagonal action. This simplicial set is contractible and its $G$-action is
free, hence it is cofibrant in $(\sset^G)_{\mathrm{proj}} := \sset^{BG}_{\mathrm{proj}}$.

\begin{lemma}\label{lem:untwist}
Let $\C\in\M^{BG}$, $K\in\sset^G$ with trivial $G$-action, and let $G$ act on
$G\times K$ by left multiplication on the first entry. Then there is an
isomorphism of $BG$-objects
\[
  \C\otimes(G\times K)\;\cong\;G\otimes\bigl(U\C\otimes K\bigr),
\]
which is natural in $\C$ and $K$.
\end{lemma}

\begin{proof}
By definition both sides have  as the underlying object $\coprod_{g\in G}(U\C\otimes K)$ in $\M$.
Let $\iota_h$ denote the inclusion of the summand indexed by $h$. On the left
the action is diagonal, so $g$ carries the $h$-th summand to the $gh$-th by
$\rho^{\C}_g\otimes\id_K$; on the right the action is free, so $g$ carries the
$h$-th summand to the $gh$-th by the identity. Define an isomorphism  $\varphi$ component-wise  by
\[
  \varphi_h:=\iota_h\circ\bigl(\rho^{\C}_{h^{-1}}\otimes\id_K\bigr).
\]
Then it is invertible in each component, so  that $\varphi$ is an isomorphism in $\M$. To see that it is equivariant, compute
 \[ \varphi\circ\bigl(\rho^{\C}_{g}\otimes\id_K\bigr)
  =\iota_{gh}\circ\bigl(\rho^{\C}_{(gh)^{-1}}\rho^{\C}_{g}\otimes\id_K\bigr)
  =\iota_{gh}\circ\bigl(\rho^{\C}_{h^{-1}}\otimes\id_K\bigr).
\]
which is the right-hand action applied to $\varphi_h$.
\end{proof}

\begin{proposition}\label{prop:cofrep}
Let $\M$ satisfy Hypotheses~\ref{hyp:main} and let $\C \in \M^{BG}$. Then
\[
  \C\otimes EG \longrightarrow \C\otimes *=\C
\]
is a cofibrant replacement in $\M^{BG}_{\mathrm{proj}}$.
\end{proposition}

\begin{proof}
Let us first show  that the map is a weak equivalence. By \ref{M2} the object $U\C$ is cofibrant, so
the pullback axiom of simplicial model categories applied to $\emptyset\to U\C$ shows that $U\C\otimes(-):\sset\to\M$ is a left Quillen functor. Every simplicial set is cofibrant, so by Ken Brown's lemma \cite{brown1973abstract}, this functor preserves all weak equivalences and applying it to $EG\to *$ yields an underlying weak equivalence, and thus a weak equivalence in the projective model structure. 

For cofibrancy, write $EG$ as a cell complex built from the free cells
$G\times\partial\Delta[q]\hookrightarrow G\times\Delta[q]$, which is possible because the
action on $EG$ is free. The functor $\C\otimes(-)$ preserves colimits, so $\C\otimes EG$ is obtained by attaching the maps
$\C\otimes(G\times\partial\Delta[q])\to \C\otimes(G\times\Delta[q])$. By
Lemma~\ref{lem:untwist} these are the maps  $G\otimes(U\C\otimes\partial\Delta[q])\to G\otimes(U\C\otimes\Delta[q])$, which are free
$G$-objects on cofibrations of $\M$ by \ref{M2} and the pullback axiom of simplicial model categories, therefore projective cofibrations.
\end{proof}

Now that we know how to make cofibrant replacements, we can finely appropriately define equivariant functors.  

\begin{definition}\label{def:correctZ2}
For $\C, \D$ a $G$-symmetric monoidal $n$-fold complete Segal space, set
\[
\Fun^{\otimes}_{G}\bigl(\C,\D)
:= \Map_{\M_n^{BG}}\bigl(\C\otimes EG,R\D),
\]
in $\M^{BG}_{\mathrm{proj}}$ where $\M_n= \CSS^{\otimes}_n$. A $G$-equivariant symmetric monoidal functor is a $0$-simplex of this simplicial set. In case the domain is a bordism category equipped with the $\zz/2$-action of Theorem \ref{thm: equivZ2bord}, then we call it a \emph{reflection TQFT}. 
\end{definition}

\begin{definition}\label{def:hfp}
For $Y$ a $G$-simplicial set whose underlying simplicial set is a Kan complex, define its \emph{homotopy fixed points} to be
\[Y^{hG}:=\Map_{\sset^G}(EG,Y),\]
so that $(Y^{hG})_{n}=\Hom_{\sset^G}(EG\times\Delta[n],Y)$ where $\Delta[n]$ is equipped with the
trivial action. 
\end{definition}

Now that all the pieces are in place we can prove Theorem \ref{thm:main}.

\begin{proof}[Proof of Theorem \ref{thm:main}]
Let $\M_n=\CSS^\otimes_n$, and let $R$ be the functorial fibrant replacement. Since $G$ acts by automorphisms, 
 $R\D$ is an object of $\M^{BG}_n$ and $\D\ra R\D$ is an equivariant weak equivalence. As
$U(R\D)$ is fibrant in $\M_n$, the object $R\D$ is fibrant in $(\M_n^{BG})_{\mathrm{proj}}$. By
Proposition~\ref{prop:cofrep} the object $\C\otimes EG$ is cofibrant, so $\FunsmG(\C,\D)=\Map_{\M_n^{BG}}(\C\otimes EG,\ R\D)$.

We check this identity levelwise. Since $\Delta[p]$ has the trivial action,
associativity of the tensoring gives $(\C\otimes EG)\otimes\Delta[p]=\C\otimes(EG\times\Delta[p])$
as $BG$-objects, and therefore
\[
  \Map_{\M_n^{BG}}\bigl(\C\otimes EG,\ R\D\bigr)_{p}
  =\Hom_{\M_n^{BG}}\bigl(\C\otimes(EG\times\Delta[p]),\ R\D\bigr)
  \cong\Hom_{\sset^G}\bigl(EG\times\Delta[p],\ \underline{\Map}(\C,R\D)\bigr)
\]
by the adjunction  in \eqref{eq:adj} applied with $K=EG\times\Delta[p]$. The right-hand side is
by definition the set of $p$-simplices of $\Map_{\sset^G}(EG,\underline{\Map}(\C,R\D))$. Hence
\[
  \FunsmG(\C,\D)\cong\Map_{\sset^G}\bigl(EG,\ \underline{\Map}(\C,R\D)\bigr)
   =\Funsm(C,D)^{hG},
\]
where the last equality  comes from  Definitions~\ref{def:funsm} and \ref{def:hfp}. The simplicial set
$\underline{\Map}(\C,R\D)=\Map_{\M}(U\C,U R\D)$ is a Kan complex by the pullback axiom of simplicial model categories. The $G$-action is the conjugation action of Construction~\ref{con:enrich}. 
\end{proof}

As an immediate corollary we get the following.

\begin{theorem}\label{z2fixedisreflection}
    Let $\C$ be a $\zz/2$-symmetric monoidal $n$-fold complete Segal space. Then 
    \[
\Fun^\otimes_{\zz/2}(\cob_n(H_n),\C) \simeq \Fun^\otimes(\cob_n({H_n}),\C)^{h\zz/2}.
    \]
    where the action on the right is the conjugation $g\cdot f= \rho_g^\D \circ f  \circ (\rho ^\C_g)^{-1}$.
\end{theorem}

We register a useful computational lemma.

\begin{lemma}\label{lem:unwind}
Let $Y$ be a Kan complex equipped with an action of $\zz/2 = \{1,g\}$. Then a $0$-simplex of
\[
Y^{h\zz/2} = \Map_{\mathcal{S}^{\zz/2}}\bigl(E\zz/2,\ Y\bigr)
\]
is equivalent to a family
\[
y_{(g_1, \dots, g_n)} \in Y_n, \ \ \ \ (g_1 ,\dots, g_n) \in (\zz/2)^n, \ \ \ n \geq 0,
\]
compatible with  face and degeneracy maps. In the first three degrees this data is:
\begin{enumerate}
\item[$n=0$:] a vertex $Z \in Y_0$;
\item[$n=1$:] an edge $\eta := y_{(g)} \in Y_1$ with $d_1\eta = Z$ and $d_0\eta = g\cdot
Z$, together with $y_{(1)} = s_0 Z$;
\item[$n=2$:] a $2$-simplex $y_{(g,g)} \in Y_2$ with
\[
d_2\,y_{(g,g)} = \eta, \qquad d_0\,y_{(g,g)} = g\cdot\eta, \qquad
d_1\,y_{(g,g)} = s_0 Z,
\]
the remaining $y_{(g_1,g_2)}$ with some $g_i = 1$ being degenerate.
\end{enumerate}
The family continues in every degree $p \geq 3$.
\end{lemma}

\begin{proof}
The action on $E\zz/2$ is free, so a orbit of $n$-simplices contains exactly one
element of the form $(1,h_1,\dots,h_p)$, and an equivariant map $\varphi \colon E\zz/2
\ra Y$ is determined by its values on them.  Setting
\[
y_{(g_1,\dots,g_n)} \;:=\; \varphi\bigl(1,\ g_1,\ g_1g_2,\ \dots,\ g_1g_2\cdots
g_n\bigr)
\]
gives the stated family. At low levels, $d_0(1,g) = (g)$ and $d_1(1,g) = (1)$,
so $d_0\eta = \varphi(g) = g\cdot\varphi(1) = g\cdot Z$ and $d_1\eta = Z$, and
$(1,1) = s_0(1)$ gives $y_{(1)} = s_0Z$. At level two, $g^2 = 1$ gives
$y_{(g,g)} = \varphi(1,g,1)$, whose faces are $(g,1)$, $(1,1)$ and $(1,g)$; the first
is $g \cdot \varphi(1,g^{-1}) = g\cdot\eta$ since $g^{-1} = g$, the second is $s_0Z$,
and the third is $\eta$.
\end{proof}

\subsection{The target category}\label{sec:target}

We now have a good homotopy theory of reflection theories theories.  As stated in Theorem \ref{thm: equivZ2bord}, we already have a desirable $\zz/2$-action on the domain. Now we need $\zz/2$-structure on a desirable target. In this section we construct the action in one class of targets, the categories of $E_n$-algebras of \cite{scheimbauer2014factorization}. 

\begin{definition}\label{def:prefac}
Let $M$ be a topological space and $(\C,\otimes,\1)$ a symmetric monoidal category. A
\emph{prefactorization algebra} $\mathcal{F}$ on $M$ valued in $\C$ consists of an object
$\mathcal{F}(U) \in \C$ for every open $U \subseteq M$, a structure map
\[
m^V_{U_1,\dots,U_k} \colon \mathcal{F}(U_1)\otimes\dots\otimes\mathcal{F}(U_k) \lra
\mathcal{F}(V)
\]
for every finite collection of pairwise disjoint opens $U_1,\dots,U_k \subseteq V$, and,
for $k = 0$, a unit map $\1 \ra \mathcal{F}(V)$, subject to unitality, equivariance for
the symmetry of $\otimes$, and associativity. To make $\mathcal{F}$ into a \emph{factorization algebra} we need to satisfy a cosheaf condition, see \cite{primerFactAyala} for details.
\end{definition}

There is an equivalence between $E_n$-algebras and locally constant factorization algebras on $\rr^n$ \cite{lurie2017higher,ginot2015notes}, which justify using them as models for our $E_n$-algebras.

In \cite{scheimbauer2014factorization}, the author defines  $\Alg_n(\C)$ as an $n$-fold complete Segal
space whose level $(k_1, \dots, k_n)$ consists of factorization algebras on $\rr^n$  with appropriate local constancy requirements. 

We now proceed to show that if $\C$ is equipped with symmetric monoidal $\zz/2$-action, then so  $\Alg_n(\C)$ is a $\zz/2$-symmetric monoidal $n$-fold complete Segal space.

\begin{proposition}\label{lem:targetaction}
Let $\tau: \C \ra \C$ denote the $\zz/2$-action symmetric monoidal action. Then post-composition with $\tau$  preserves locally constant factorization algebras on $\rr^n$
and induces a  $\zz/2$-action on $\Alg_n(\C)$, that is a functor
$B\zz/2 \ra \CSS^\otimes_n$. In particular $\Alg_n(\C)$ is a $\zz/2$-symmetric monoidal
$n$-fold complete Segal space.
\end{proposition}

\begin{proof}
Define $\overline{\mathcal{F}}(U) := \tau(\mathcal{F}(U))$ with structure maps
$\tau(m^V_{U_1,\dots,U_k})$. The axioms of Definition \ref{def:prefac} are equations
between morphisms that are preserved by any symmetric monoidal functor. Local constancy is preserved because $\tau$ is an equivalence. Because $\tau^2 = \id$, the assignment is a functor out
of the one-object category $B\zz/2$, and it is symmetric monoidal because $\tau$ is.
\end{proof}

Now we proceed to define a concrete example, namely by taking $\C =s\vect_\cc$, the category of super vector spaces. This is the category whose objects are $\zz/2$-graded complex vector spaces with grading preserving linear maps.

\begin{proposition}\label{prop:conjugation}
Let $\tau \colon s\vect_\cc \ra s\vect_\cc$ send a super vector space $V$ to its complex
conjugate $\overline{V}$, that is, the super vector space with same underlying real super vector space with scalar multiplication $\lambda \cdot_{\overline{V}} v := \overline{\lambda}v$. Then $\tau$ is a $\zz/2$-action.
\end{proposition}

\begin{proof}
First note that $\overline{\overline{V}} = V$ on the nose, since conjugating twice restores the original
scalar multiplication, and the comparison $\overline{V}\otimes\overline{W} \ra
\overline{V\otimes W}$ is the identity on underlying real spaces. The only point at which
the grading enters is compatibility with the braiding $v\otimes w \mapsto
(-1)^{|v||w|}w\otimes v$, whose structure constants lie in $\{\pm 1\}\subset\rr$ and are
fixed by conjugation. The unit and associativity constraints are likewise real.
\end{proof}

\begin{corollary}
The $n$-fold space complete Segal space $s\Alg_n(\cc) := \Alg_n(s\vect_\cc)$ is is a $\zz/2$-symmetric monoidal complete $n$-fold Segal space. 
\end{corollary}

\begin{theorem}\label{t:main}
Let $\C$ be a $\zz/2$-symmetric monoidal $n$-fold complete Segal space with action
$\tau$, and let $Z$ be a $\C$-valued reflection TQFT with symmetry type $H_n$. Then for every object $Y$ of
$\cob_n(H_n)$ there is an equivalence
\[
\tau\bigl(Z(Y)\bigr) \ \simeq\ Z(\alpha Y)
\]
in $\C$, where $\alpha$ is the involution of Proposition \ref{prop:strictaction}.
\end{theorem}

\begin{proof}
By Theorem \ref{thm:main} and Lemma \ref{lem:unwind} the equivariant structure supplies
an edge $\eta$ from $Z$ to $g\cdot Z = \tau\circ Z\circ\alpha$ in the Kan complex
$\Map(\cob_n(H_n), R\C)$; that is, a map $\cob_n(H_n)\otimes\Delta[1] \ra R\C$
restricting to $Z$ and to $\tau\circ Z\circ\alpha$ at the two ends. Restricting along
$\{Y\}\times\Delta[1]$ gives a path $p$ in $(R\C)_{0,\dots,0}$ from $Z(Y)$ to
$\tau(Z(\alpha Y))$.

Since $R\C$ is Reedy fibrant, the map $(d_1,d_0)\colon (R\C)_1 \ra (R\C)_0\times(R\C)_0$
is a fibration; lifting the path $(\mathrm{const}_{Z(Y)}, p)$ starting at the vertex
$\id_{Z(Y)}$ produces a $1$-morphism $Z(Y) \ra \tau(Z(\alpha Y))$. It is an equivalence
because $\id_{Z(Y)}$ is one and the space of equivalences is a union of connected
components of $(R\C)_1$. Applying $\tau$ and using $\tau^2 = \id$ gives the statement.
\end{proof}

\begin{corollary}\label{cor:levels}
Let $Z$ be as in Theorem \ref{t:main}. Then for every $k_1,\dots,k_n$ and every simplex
$x$ of $\cob_n(H_n)$ at that level there is an equivalence
\[
\tau\bigl(Z(x)\bigr) \ \simeq\ Z(\alpha x)
\]
in $\C$. In particular this holds for closed $n$-manifolds, regarded as
$n$-endomorphisms of the monoidal unit.
\end{corollary}

\begin{proof}
By the proof of Theorem \ref{t:main} used only that $\eta$ is a map
$\cob_n(H_n)\otimes\Delta[1] \ra R\C$. Restricting along $\{x\}\times\Delta[1]$ gives a
path in $(R\C)_{k_1,\dots,k_n}$ between $Z(x)$ and $\tau(Z(\alpha x))$, and a path in a
Kan complex is an equivalence.
\end{proof}

\begin{remark}\label{rmk:reality}
Corollary \ref{cor:levels} applied to a closed $n$-manifold $M$ is the extended
counterpart of the relation $Z(\overline{M}) = \overline{Z(M)}$ of
\cite[\S4.1]{FreedHopkinsReflection}, and it is the main evidence that Definition
\ref{def:correctZ2} is the right notion: the reality of partition functions is derived
from the equivariance datum rather than imposed. Extracting from it a statement about a
\emph{number} requires identifying $\Omega^n\C$, the space in which $Z(M)$ takes values;
this depends on the target and is not attempted here. 

\end{remark}

\subsection{An example: Clifford algebras in dimension one}\label{sec:example}

In this section we produce an illustrative example.

\begin{definition}
    Let $\mathrm{Mor}_1(\C)$ be the category whose objects are algebras and whose morphisms from an algebra $A$ to an algebra $B$ are equivalences classes of $(A, B)$-bimodules ${}_{A}M_{B}$, where ${}_{A}M_{B}$ is equivalent to ${}_{A}M'_{B}$ iff   $M \simeq M'$. For $\C = s\vect_\cc$ we just write $\mathrm{Mor}_1$. 
\end{definition}

\begin{definition}\label{def:sop}
Let $A$ be an algebra in $s\vect_\cc$ with multiplication $m \colon A\otimes A \ra A$,
and let
\[
\sigma_{A,A} \colon A \otimes A \lra A \otimes A, \qquad
\sigma_{A,A}(a\otimes b) = (-1)^{|a||b|}\, b \otimes a
\]
be the symmetry isomorphism of $s\vect_\cc$. The \emph{super-opposite algebra}
$A^{\sop}$ is the same object of $s\vect_\cc$ with multiplication
$m\circ\sigma_{A,A}$; on homogeneous elements,
\[
a\cdot_{\sop}b \ = \ (-1)^{|a||b|}\, b\,a .
\]
\end{definition}

\begin{proposition}\cite[Proposition 3.4.3]{scheimbauer2014factorization}
    There is an equivalence of symmetric monoidal categories
    \[
\H(\Alg_1(\C)) \simeq \mathrm{Mor}_1(\C).
    \]
\end{proposition}

\begin{proposition}
    Let $\mathrm{Mor}_1$ be the $1$- super  Morita category. Then for any $A \in\mathrm{Mor}_1$, $A^\sop$ is a dual for $A$.
\end{proposition}

\begin{lemma}\label{lem:ptdual}
Let $Z: \cob_n(H_n) \ra \C$ be a TQFT and let $Y$ be an object of $\Omega^{n-1} \cob_n(H_n)$, then 
\[
Z(\alpha Y)  \simeq Z(Y)^\lor.
\]
where $Z(Y)^\lor$, as an object, is a dual for $Z(Y)$ in $\H(\Omega^{n-1}\C)$.
\end{lemma}

\begin{proof}
By \cite[8.20]{claudiaDamien}, $\H(\Omega^{n-1}\cob_n(H_n))\simeq \cob_{(n,n-1)}(H_n)$ and by \cite[4.8]{FreedHopkinsReflection}, $\alpha Y \cong Y^\lor$, so the result follows since symmetric monoidal functors preserve duals.
\end{proof}

\begin{lemma}\label{lem:alphapt}
Let $(H_n,\rho_n)$ be a symmetry type, so that $\rho_n(H_n) \supseteq SO(n)$
\cite[Def.~2.4]{FreedHopkinsReflection}, and let $\alpha$ be the involution of manifolds, induced by
a reflection $r \in O(n)\setminus SO(n)$ (see \cite{FreedHopkinsReflection}). Set $G = \rho_n(H_n) \supseteq SO(n)$ and $K = \ker(\rho_n)$. The following holds. 
\begin{enumerate}
\item[(i)] if $G = SO(n)$, there are exactly two classes $\pt_{+},\pt_{-}$, and
$\alpha\,\pt_{\pm} \cong \pt_{\mp}$;
\item[(ii)] if $G = O(n)$, there is a single class $\pt$, and
$\alpha\,\pt \cong \pt$.
\end{enumerate}
\end{lemma}

\begin{proof}
An $H_n$-structure on the point is a lift along $B\rho_n$ of the map
$* \ra BO(n)$ that classifies stable tangent bundle, therefore it is a point of the
homotopy fiber $F$ of $B\rho_n \colon BH_n \ra BO(n)$. 
Different structures are $\pi_0 F$. Using the fiber sequence
\[
  O(n) \lra F \lra BH_n \lra BO(n)
\]
we get that 
\[
  \pi_0 F  \cong \ O(n)/G, \ \ \ \  \pi_1 F  \cong  K ,
\]
the first because $G$ is the image of $\rho_n$ and the second because $K$ is its
kernel. In our case,  $O(n)/G$ has one or two  elements depending on $G = O(n)$ or $G = SO(n)$.  The involution $\alpha$ acts on structures by precomposing the classifying data
with $r$, hence on $O(n)/G$ by $gG \mapsto rgG$. In case (ii) the set is a single point and there is nothing to prove. In case (i) the map $O(n)/SO(n) \ra \zz/2$ given by the determinant is a bijection, and left multiplication by $r$ has $\det r = -1$, so it exchanges the two classes.
\end{proof}

\begin{notation}
    In light of Lemma \ref{lem:alphapt}, we will always choose an arbitrary point and called it $\pt_+$ and by consequence $\pt_- \simeq \alpha \pt_+$. Thus, even though they carry different symbols in the notation, it might very well be true that $\pt_+ \simeq \pt_-$.
\end{notation}

\begin{example}\label{ex:spin1}
For $H_1 = \Spin_1 = \{\pm1\}$ the map $\rho_1$ is trivial, so $\rho_1(H_1) = SO(1)$ and
case (i) applies: two objects $\pt_{\pm}$ interchanged by $\alpha$, each with
automorphism group the internal symmetry group $K = \ker\rho_1 = \zz/2$. By
\cite[(3.33)]{FreedHopkinsReflection}, $\hat H_1 = \Pin^{+}_1 \cong \zz/2\times\zz/2$.
For $H_1 = \Pin^{\pm}_1$ the map $\rho_1$ is onto $O(1)$ and case (ii) applies.
\end{example}

\begin{corollary}
    If $Z: \cob_n(H_1) \ra s\Alg_1(\cc)$  is symmetric monoidal, and $Z(\pt_+)=A$, then we must have $Z(\pt_-) \simeq A^\sop$
\end{corollary}

\begin{corollary}\label{cor:cliff}
If $Z$ admits a reflection structure in the sense of Definition
\ref{def:correctZ2}, and $Z(\pt_+)=A$, then there is an equivalence
\[
\overline{A} \ \simeq\ A^{\sop}
\]
in $s\Alg_1(\cc)$.
\end{corollary}

 We  will now focus on the case $Z(\pt_+)=\Cl_1$. The algebra $\Cl_1$ is not chosen arbitrarily: it is the value at the point of the parity
theory appearing in Gunningham's work on spin Hurwitz numbers \cite{gunningham}, whose partition function on spin surfaces computes $(-1)^{\mathrm{Arf}}$. That theory is two-dimensional, and the example treated here is its one-dimensional shadow.

\begin{notation}\label{not:input}
Let $Z_{\Cl_1} : \cob_1(\Spin_1) \ra \Alg_1(s\vect_\cc)$ be a symmetric monoidal functor
with
\[
Z_{\Cl_1}(\pt_+) \cong \Cl_1  :=  \cc\langle e\rangle/(e^2-1), \ \ \  |e| = 1 .
\]
\end{notation}

\begin{lemma}
    If $Z_{\Cl_1}$ admits a reflection structure, then we must have 
    \[\overline{\Cl_1} \ \simeq\ \Cl_1^{\sop}.\]
\end{lemma}

\begin{lemma}\label{lem:clsop}
We compute that
$\Cl_1^{\sop} \cong \Cl_{-1} := \cc\langle e\rangle/(e^2+1)$.
Over $\cc$ we have  an isomorphism $\Cl_{-1} \cong \Cl_1$ via $e \mapsto ie$, and $\overline{\Cl_1}\cong\Cl_1$.
\end{lemma}

\begin{proof}
Since $|e| = 1$, Definition \ref{def:sop} gives $e\cdot_{\sop}e = (-1)^{1\cdot1}ee = -1$ and 
whereas $e\cdot_{\op}e = ee = +1$. Since $\Cl_1$ is generated by $e$ this determines both algebras. The substitution $e\mapsto ie$ carries $e^2 = -1$ to $e^2 = 1$. For the last claim, $\Cl_1$ has the real form $\rr\oplus\rr e$, and the antilinear map fixing $1$ and
$e$ is an isomorphism of $\cc$-algebras $\overline{\Cl_1}\cong \Cl_1$.
\end{proof}

\begin{definition}\label{def:stellar}
An \emph{antilinear stellar structure} on an  complex super-algebra  $A$  is an antilinear bijection $\theta$ of $A$ with $\theta(1) = 1$ and
\begin{equation}\label{eq:theta}
    \theta(ab) \ = \ (-1)^{|a||b|}\,\theta(b)\,\theta(a).
\end{equation}
This is equivalent to an isomorphism $\overline{A} \cong A^{\sop}$ of algebras in $s\vect_\cc$. Stellar algebras were introduced by Schommer-Pries in \cite{SchommerPries} and we use the $\cc$-antilinear version of of M\"uller and Stehouwer defined in  \cite[\S5.2]{MullerStehouwer}.
\end{definition}

\begin{proposition}\label{prop:clifford}
For $A = \Cl_1$ there are exactly two stellar structures,  $\theta_{\pm}(e) = \pm ie$.  They both satisfy $\theta_\pm^2 = \id$. Furthermore,  the pair is a torsor under $\Aut(\Cl_1) =\{\id,(-1)^F\}$.
\end{proposition}

\begin{proof}
 $\overline{\Cl_1}$ has the same underlying
set and product with conjugated scalars, so a $\cc$-algebra map out of it is an antilinear map. Note that $\theta$ is  grading preserving since it is a morphism in $s\vect_\cc$.
The odd part of $\Cl_1$ is one dimensional over $\cc$, so $\theta(e) = \lambda e$
determines $\theta$. We can apply \eqref{eq:theta} with $a = b = e$,
\[
1 = \theta(1) = \theta(e\cdot e) = (-1)^{1\cdot1}\theta(e)\theta(e) = -\lambda^2 e^2
= -\lambda^2,
\]
so $\lambda^2 = -1$ and $\lambda = \pm i$; conversely both satisfy \eqref{eq:theta}. In
particular $|\lambda| = 1$, so by antilinearity
$\theta^2(e) = \overline{\lambda}\theta(e) = |\lambda|^2 e = e$, and $\theta^2 = \id$ on
the even part since $\theta(1) = 1$.

A grading-preserving automorphism $\phi$ of $\Cl_1$ has $\phi(e) = \mu e$ with
$\mu^2 = 1$, so $\mu$ is real and $(\phi^{-1}\theta\phi)(e) =
\overline{\mu}\lambda\mu^{-1}e = \lambda e$; hence $\lambda$ is an isomorphism invariant.
Finally $(-1)^F(e) = -e$, so $\theta\circ(-1)^F$ sends $e\mapsto -\lambda e$, exchanging
$\theta_+$ and $\theta_-$.
\end{proof}

Notice that so far we have only talked about morphisms in $s\vect_\cc$, but the our theory lives in the Morita setting, so we need to talk about how these morphisms are actually coming into play.

\begin{notation}\label{not:twist}
Consider  $\varphi \colon A \ra B$ an isomorphism in $s\vect_\cc$, let
$M_{\varphi}$ denote the $(A,B)$-bimodule with underlying object $B$, left
action $a\cdot m := \varphi(a)m$ and right action the regular one. Each
$M_{\varphi}$ is invertible, with inverse $M_{\varphi^{-1}}$, and
\[
  M_{\psi\circ\varphi} \cong  M_{\varphi}\otimes_{B} M_{\psi}
\]
using $m\otimes m' \mapsto \psi(m)m'$. If $A = B$ this provides a morphism
$\Aut(A) \ra \Pic^{s}(A) : = \Aut_{\mathrm{Mor}_1}(A)$. An element in the image of this morphism is called an \emph{inner automorphism}.
\end{notation}

\begin{lemma}\label{lem:twistiso}
Let $\varphi,\psi : A \ra B$ be isomorphisms. Then $M_{\varphi} \cong M_{\psi}$ if and only if $\psi\varphi^{-1}$ is conjugation by an even unit of $B$. Thus the kernel of
$\Aut(B) \ra \Pic^{s}(B)$ consists of the automorphisms that are inner via an even unit.
\end{lemma}

\begin{proof}
An isomorphism $f :  M_{\varphi} \ra M_{\psi}$ is in particular an even bijection which is
right $B$-linear, thus $f(m) = cm$ for $c := f(1)$, necessarily an even unit.
Left linearity gives $c\varphi(a) = \psi(a)c$ for all $a \in A$. Writing
$b = \varphi(a)$ we get
\[
  \psi\varphi^{-1}(b)  =  cbc^{-1}  \ \ \ \ (b \in B).
\]
Conversely any such $c$ defines an isomorphism. 
$f$ is even.
\end{proof}

\begin{proposition}\label{prop:thetadistinct}
The structures $\theta_{\pm}$ of Proposition \ref{prop:clifford} are not
isomorphic as $1$-morphisms in $s\Alg_1(\cc)$. The same holds for $c_{\pm}$ of
Proposition \ref{prop:pin}.
\end{proposition}

\begin{proof}
Both $\theta_{\pm}$ are isomorphisms $\overline{\Cl_1} \cong \Cl_1^{\sop}$, so they
determine $1$-morphisms $M_{\theta_{\pm}}$, and $\theta_-\theta_+^{-1} = (-1)^F$ by Proposition \ref{prop:clifford}. The even part of $\Cl_1^{\sop}$ is $\cc$, so every even unit is central and conjugation  by it is the identity and  $(-1)^F \neq \id$. By Lemma \ref{lem:twistiso}, $M_{\theta_+}$ cannot  be isomorphic to $ M_{\theta_-}$. The argument for $c_{\pm}$ is the same.
\end{proof}

\begin{remark}
    It is well known that $\Pic^s(\Cl_1) = \zz/2$, thus $\theta_{\pm}$ actually gives us all the automorphisms of $\Cl_1$.
\end{remark}

\begin{remark}\label{rmk:etacomponents}
Now we can see how $\theta$ sits inside the equivariant structure. By abuse of notation we use $\theta_{\pm}$ and the corresponding bimodules interchangeably. The edge $\eta$ coming from  Lemma \ref{lem:unwind} has a component at each object, and $\alpha\,\pt_\pm = \pt_\mp$ by Lemma \ref{lem:alphapt}(i), so
\[
\eta_{\pt_+} \colon Z(\pt_+) \lra \overline{Z(\pt_-)}, \qquad
\eta_{\pt_-} \colon Z(\pt_-) \lra \overline{Z(\pt_+)} .
\]
 The map $\eta_{\pt_+}$ is one of the $\theta$ from
Proposition \ref{prop:clifford} and $\eta_{\pt_-}$  is at first independent data. The degree-two simplex $y_{(g,g)}$ relates them  as it must trivialize the composite
$Z \Rightarrow \tau Z\alpha \Rightarrow \tau^2 Z\alpha^2 = Z$, whose component at $\pt_+$
is $\overline{\eta_{\pt_-}}\circ\theta$, so it forces
$\eta_{\pt_-} = \overline{\theta^{-1}}$, that is $\theta^2 = \id$. For $\Cl_1$ this is no
condition at all since by the proof of Proposition \ref{prop:clifford} this always holds.
\end{remark}

\begin{remark}\label{rmk:noinvariant}
Even though the theories are therefore non isomorphic they may still be indistinguishable.  $\theta$ is an
isomorphism  and  the theory $Z \circ \alpha \simeq \overline{Z}$ obtained by precomposition with $\alpha$ is
naturally isomorphic to conjugation of the theory, and
$\overline{\theta_{\pm}} = \theta_{\mp}$ so that the two structures have
complex-conjugate partition functions. On $\mathrm{Spin}_1$-manifolds these may
agree, in which case no invariant computed by $Z_{\Cl_1}$ separates them
even though they are not equivalent.
\end{remark}

\subsubsection*{The $\Pin^\pm$ case}

\begin{proposition}\label{prop:pin}
Let $H_1 = \Pin^{\pm}_1$ and let $Z \colon \cob_1(H_1) \ra s\Alg_1(\cc)$ be a
symmetric monoidal functor admitting a reflection structure, with
$A := Z(\pt)$. Then there is an equivalence
\[
\overline{A}  \simeq A,
\]
which is realized by an isomorphism of algebras, this isomorphism is an antilinear
automorphism $c$ of $A$ satisfying  moreover $c^2 = \id$ so that $A$ is the
complexification of the real form $A^{c}$. For $A = \Cl_1$ there are exactly two antilinear automorphisms, $c_{\pm}(e) = \pm e$,
both of which satisfy $c_\pm^2 = \id$,  The pair is again a 
torsor under $\{\id,(-1)^F\}$.
\end{proposition}

\begin{proof}
By Lemma \ref{lem:alphapt}(ii), we have that
$\alpha\pt \cong \pt$, so Theorem \ref{t:main} reads
$\overline{A} \simeq Z(\alpha\,\pt) \simeq A$. Unwinding, an isomorphism
$\overline{A} \cong  A$ of algebras in $s\vect_\cc$ is an antilinear grading-preserving
bijection $c$ with $c(1) = 1$ and $c(ab) = c(a)c(b)$.

For $A = \Cl_1 = \cc\langle e\rangle/(e^2-1)$, the odd part is one dimensional over $\cc$, thus $c(e) = \mu e$ and dertermines
$c$. Since $e^2 = 1$ we have $\mu^2 = 1$, hence $\mu = \pm1$, and
antilinearity gives $c^2(e) = \overline{\mu}\mu e = |\mu|^2 e = e$, so  that $c^2 = \id$. One checks that $c_- = c_+\circ(-1)^F$.
\end{proof}

\begin{remark}\label{rmk:dichotomy}
The previous computations show that the equivariance data imposed by a reflection theory correspond to algebraic structure on the value of the algebra on the point. Succinctly we have the following
\begin{center}
\begin{tabular}{lll}
\hline
$H_1$ & $\alpha$ on objects & structures on $A$ \\
\hline
$\Spin_1$ & $\pt_{+}\leftrightarrow\pt_{-}$ & $\theta_{\pm}(e) = \pm ie$, stellar structure\\
$\Pin^{\pm}_1$ & $\pt \mapsto \pt$ & $c_{\pm}(e) = \pm e$, real forms \\
\hline
\end{tabular}
\end{center}

The table just records what the \emph{reflection structure} supplies, not which relation holds. For instance, we still need $A \simeq A^\sop$ since the point is self-dual in this case. 

By Lemma \ref{lem:alphapt}, what changes the behavior as seen above is wether $\rho_1(H_1)$ is $SO(1)$ or $O(1)$. In the first case $\alpha$ interchanges the two objects and Lemma \ref{lem:ptdual} converts the equivalence into one with $\Cl_1^{\sop}$, so the sign appears, and in the second case $\alpha$ fixes the object and no opposite algebra is involved. 
\end{remark}

\bibliographystyle{alpha}
\bibliography{References}

\end{document}